\documentclass{amsart}
\usepackage[utf8]{inputenc}
\usepackage{amsmath}
\usepackage{amsfonts}
\usepackage{amssymb}
\usepackage{amsthm}
\usepackage{verbatim}
\usepackage[]{graphicx}
\usepackage[ruled,vlined]{algorithm2e}
\SetKw{Break}{break}
\usepackage{xcolor}
\usepackage{todonotes}
\usepackage{hyperref}
\usepackage[style=alphabetic,giveninits,citestyle=alphabetic,doi=false,isbn=false,url=false,eprint=false,maxnames=50,sorting=nyt]{biblatex}
\usepackage{amsaddr}

\newtheorem{theorem}{Theorem}[section]

\newtheorem{lemma}[theorem]{Lemma}
\newtheorem{proposition}[theorem]{Proposition}
\newtheorem{definition}[theorem]{Definition}
\newtheorem{remark}[theorem]{Remark}
\newtheorem{problem}[theorem]{Problem}

\newtheorem{example}[theorem]{Example}

\newcommand{\dee}{{\rm d}}
\newcommand{\normal}{{\rm n}}

\newcommand{\R}{\mathbb{R}}

\newcommand{\id}{{\rm id}}

\DeclareMathOperator{\Lip}{Lip}
\DeclareMathOperator{\LipO}{Lip_1}
\DeclareMathOperator{\Div}{div}

\DeclareMathOperator*{\argmin}{arg\,min}
\newcommand{\measurerestr}{
  \,\raisebox{-.127ex}{\reflectbox{\rotatebox[origin=br]{-90}{$\lnot$}}}\,
}

\title[Shape Optimisation and Optimal Transport]{Shape Optimisation and a link to Optimal Transport}
\title[Shape Optimisation with $W^{1,\infty}$]{Shape Optimisation with $W^{1,\infty}$:  A connection between the steepest descent and  Optimal Transport}
\author[P. J. Herbert]{Philip J. Herbert}
\address{Maxwell Institute for Mathematical Sciences, Department of Mathematics, Heriot-Watt University, Edinburgh EH14 4AS, United Kingdom}
\email{p.herbert@hw.ac.uk}
\thanks{This work is part of the project P8 of the German Research Foundation Priority Programme 1962, whose support is gratefully acknowledged by the author.  The author also acknowledges the support of EPSRC (grant EP/W005840/1).}
\keywords{PDE constrained shape optimisation, $W^{1,\infty}$ descent, optimal transport, Lipschitz functions}

\date{\today}

\begin{document}

\maketitle

\begin{abstract}
    In this work, we discuss the task of finding a direction of optimal descent for problems in Shape Optimisation and its relation to the dual problem in Optimal Transport.
    This link was first observed in a previous work which sought minimisers of a shape derivative over the space of Lipschitz functions which may be closely related to the $\infty$-Laplacian.
    We provide some results of Shape Optimisation using this novel Lipschitz approach, highlighting the difference between the Lipschitz and $W^{1,\infty}$ semi-norms.
    After this, we provide an overview of the necessary results from Optimal Transport in order to make a direct link to the optimisation of star-shaped domains.
    Demonstrative numerical experiments are provided.
\end{abstract}
\section{Introduction}
The task of optimising shapes is a classical problem which has lead to much research.
Many industries are concerned with having an efficient shape, in aerospace this might be objects with minimal drag, in civil design this might be a structure which supports some load which has minimal mass.
The fact that there are many applications in industry necessitates the efficient solution, or approximation of optimal shapes.
In order to approximate optimal shapes in practical settings, one will often use algorithms which utilise the so-called steepest descent.
The steepest descent depends on the topology.
Frequently, a Hilbertian topology is utilised.
A method using steepest descent in Hilbert spaces appears in \cite{ADJ21} and a comparison of numerical approximation of Hilbertian shape gradients are presented in \cite{HPS15}.
The Hilbertian setting is often, in the computational setting, cheap to calculate.
A downside however is, choosing an appropriate Hilbert space.
For two particularly interesting choices of Hilbert spaces, we refer to \cite{IglSturWec18} who use a nearly-conformal map and \cite{EigStu18} which uses Reproducing Kernel Hilbert spaces.
One wishes for the chosen space to be continuously embedded within continuous functions, otherwise the shape derivative need not make sense.
More specifically, one should require that the chosen space is embedded within Lipschitz functions, this means that any small perturbation is invertible, which need not be the case for functions which are only continuous.
These restrictions on Hilbert spaces can often require that one considers higher order Sobolev spaces, which may result to seeking a direction in too small of a space.

Shape Optimisation using functions in Banach spaces is a new area of research.
The method proposed in \cite{DecHerHin21} was to consider the Banach space $W^{1,\infty}$.
In this space, one wishes to impose a unit bound on the semi-norm $V\mapsto \| DV\|_{L^\infty}$, which leads to a computationally challenging non-smooth problem.
In \cite[Section 2.3]{DecHerHin21} it was noticed that, under certain regularity assumptions, the problem of finding a direction of steepest descent for the shape derivative could be identified with the dual of an Optimal Transport problem.
This link was exploited for the numerical experiments in that work.
It is hoped that the link between these two previously unrelated problems may aid the development of more efficient solvers for use in finding the direction of steepest descent in Shape Optimisation problems.

This article is mainly concerned with the further study of Lipschitz, or $W^{1,\infty}$, functions in Shape Optimisation, where the main contribution is to rigorously demonstrate a link between the dual problem of Optimal Transport and the direction of steepest descent for Shape Optimisation with star-shaped domains.
We will also contribute to Shape Optimisation using functions in Banach spaces by demonstrating the existence of directions of steepest descent for shape problems with the Lipschitz topology under conditions which are reasonable in many applications.
While we have stated that we will link the direction of steepest descent for the optimisation of star-shaped domain to an Optimal Transport problem, it is not always the case that one may use star-shaped domains.
In the case that one cannot restrict to star-shaped domains, finding computational solutions to the steepest descent problem remains challenging, as such strategies in which to approximate the steepest descent in a continuous setting are provided, whereby the computational solutions may be somewhat easier to construct.
These approximations are shown to achieve descent close to that of the steepest descent in the Lipschitz topology.
One such approximation utilises the $p$-Laplace operator, which has appeared in \cite{MulKulSie21,MulPinRun22} as an approximation to the direction of steepest descent in the Lipschitz topology.

In recent years, both the subjects of Shape Optimisation and Optimal Transport have undergone a renaissance, see \cite{SSW15,SSW16,SW17,GHHK15,GHKL18,HUU20,HSU21,HenPie18,HP15} for Shape Optimisation and \cite{ClaLorMah21,KetMon18,MonSuh21, LorMah21, LorManMey19, LorMahMey22} for Optimal Transport, particularly for machine learning applications \cite{ArjChiBot17,RolCutPey16,CarCutOud17}, while having an established history, see \cite{BFLS97,GM94,MS76,S80} for shape optimisation and \cite{Vil21,AmbGilgSav05,AmbGil13,BogKol12,McCGui11} for Optimal Transport.
While the two subjects remain moderately separate, recent work has demonstrated that there is a link between them.
This link relates to the problem of finding a direction of steepest descent for the shape derivative.

We refer to the works \cite{DelZol11}, \cite{SZ92} and \cite{ADJ21} for details and references on Shape Optimisation.
For an applied mathematician's perspective on Optimal Transport we refer to \cite{San15} and to \cite{PeyCut19} for any computational aspects of Optimal Transport.

\subsection*{Outline}
We begin in Section \ref{sec:ShapeOpt} by giving a brief overview of the problem of Shape Optimisation.
In this Section, the space of Lipschitz functions are defined and discussed, we introduce the problem of finding a direction of steepest descent, and demonstrate the existence of such a direction.
It is this problem, of the direction of steepest descent, in which we observe a possible link between Shape Optimisation and Optimal Transport.
To prepare to show this link, in Section \ref{sec:OptTran}, we introduce many of the fundamental ideas from Optimal Transport.
This is followed by Section \ref{sec:StarShaped} which discusses the problem of optimising a star-shaped domain, whereby one may combine the previous results to provide the link between the direction of steepest descent and the solution of the Optimal Transport problem.
This is followed by analysing how well the steepest descent may be approximated in a two-dimensional star-shaped setting.
In Section \ref{sec:Experiments} we conclude by giving some demonstrative numerical examples.

\section{Shape Optimisation}\label{sec:ShapeOpt}
A general problem of Shape Optimisation may be posed as follows.
\begin{problem}\label{prob:GeneralShapeOptim}
    Given $\mathcal{S}$ a collection of domains and a proper function $\mathcal{J}\colon \mathcal{S}\to \R$, find $\Omega \in \mathcal{S}$ such that
    \begin{equation}
        \mathcal{J}(\Omega) = \inf_{\Omega' \in \mathcal{S}}\mathcal{J}(\Omega').
    \end{equation}
\end{problem}
We are interested in PDE constrained Shape Optimisation.
The above problem is said to be a PDE constrained Shape Optimisation problem when $\mathcal{J}$ takes the form $\mathcal{J}(\Omega) = J^*_\Omega (u_\Omega)$, where $u_\Omega$ is the solution to some PDE on the domain $\Omega$ and $J^*_\Omega$ is a functional which depends on $\Omega$.
\begin{example}\label{ex:PoissonExample}
    One of the most simple examples is the Poisson problem with tracking-type functional.
    This is represented by taking
    \begin{equation}\label{eq:quadraticEnergy}
        J^*_\Omega(v) := \frac{1}{2}\int_\Omega (v-Z)^2
    \end{equation}
    and the function $u_\Omega \in H^1_0(\Omega)$ satisfies
    \begin{equation}
        \int_\Omega \nabla u_\Omega \cdot \nabla \eta = \int_\Omega F \eta \quad \forall \eta \in H^1_0(\Omega)
    \end{equation}
    for $Z$ and $F$ given functions.
\end{example}
In the context of seeking a steepest descent, the above example was considered in \cite{DecHerHin21} under the restriction of domains being star-shaped.

A particular difficulty of Shape Optimisation is the highly non-linear nature of the problems.
A typical metric on the space of shapes is the so-called Hausdorff complimentary metric.
It is possible to show the existence of minimisers of Problem \ref{prob:GeneralShapeOptim} for a variety of $\mathcal{J}$ and $\mathcal{S}$.
Generally, showing existence will utilise a representation of shapes in $\mathcal{S}$ by some sort of function - for example, the indicator function of the domain.
In \cite{DecHerHin21}, existence of solutions to the problem detailed in Example \ref{ex:PoissonExample} was shown for the case of $\mathcal{S}$ comprising of domains which are star-shaped with respect to the origin, are contained within a given hold-all domain and satisfy a given cone condition.
A result similar to this, in the case of domains which need not be star-shaped, is given in Chapter 8, Theorem 4.1 of \cite{DelZol11}.

While minimisers may exist, methods for showing the existence will generally utilise compactness properties of the aforementioned representation of shapes.
The lack of a clear linear (or convex) structure on the space of shapes means that finding a solution (or an approximation) is rather difficult.
To find minimisers of $\mathcal{J}$ over $\mathcal{S}$, one might consider approaches which use the derivative, in particular steepest descent methods.
As such we are interested in the case that $\mathcal{J}$ is differentiable in some sense.
Before we define what is meant by a derivative in this setting of shapes, it is useful to establish notation and discuss properties of Lipschitz functions.

\subsection{Lipschitz functions}
For the moment, let $\Omega$ be a bounded domain.
Define the space of Lipschitz functions as
\begin{equation}
    C^{0,1}(\Omega;\R^d) := 
    \left\{ V \colon \Omega \to \R^d :  \sup_{x,y \in \Omega, x\neq y}\frac{|V(x)-V(y)|}{|x-y|} < \infty \right\},
\end{equation}
where $|\cdot|$ is the standard Euclidean norm.
Notice that this definition of Lipschitz functions uses the extrinsic metric in $\R^d$.
A function $V \in C^{0,1}(\Omega;\R^d)$ satisfies $V \in \LipO(\Omega;\R^d)$ if ${|V(x)-V(y)|}\leq {|x-y|} $ for all $x,y \in \Omega$.
Functions which are in $\LipO(\Omega;\R^d)$ satisfy the following property:
\begin{lemma}
Given $t \in (-1,1)$ and $V \in \LipO(\Omega;\R^d)$, define the map $\Phi_t \colon x \mapsto x + t V(x)$.
It holds that $\Phi_t$ is invertible on its image, furthermore its inverse in Lipschitz, that is to say $\Phi_t$ is a bi-Lipschitz map.
In particular, $\Phi_t$ is a homeomorphism.
\end{lemma}
\begin{proof}
    To show that $\Phi_t$ is invertible on its image, it is sufficient to show that it is injective.
    Let $x,y \in \Omega$, suppose that $\Phi_t(x) = \Phi_t(y)$, that is $x+ tV(x) = y + tV(y)$, then
    \begin{equation}
        |x-y| = |t| |V(x)-V(y)| \leq |t| |x-y|,
    \end{equation}
    therefore $x = y$.
    To see that $\Phi^{-1}_t$ is Lipchitz, consider $\hat{x},\hat{y} \in \Phi_t(\Omega)$ which satisfy $\hat{x} := x + tV(x)$ and $\hat{y}:= y+tV(y)$ for some $x,y \in \Omega$.
    Then
    \begin{equation}
        |\Phi_t^{-1}(\hat{x}) -\Phi_t^{-1}(\hat{y})|
        =
        |x-y|
        =
        |\hat{x} - tV(x)- \hat{y} + tV(y)|
        \leq
        |\hat{x}-\hat{y}| + |t| |V(x)-V(y)|,
    \end{equation}
    therefore $(1-|t|)|\Phi_t^{-1}(\hat{x}) -\Phi_t^{-1}(\hat{y})| \leq |\hat{x}-\hat{y}|$, so $\Phi_t^{-1}$ is Lipschitz with Lipschitz constant $\frac{1}{1-|t|}$.
\end{proof}

\subsection{Shape differentiability}
We now define what it means to be shape-differentiable, here we use the perturbations of the identity definition.
\begin{definition}
    Let $\Omega \in \mathcal{S}$ and $V \in \LipO(\Omega;\R^d)$.
    Writing $\Omega_t(V) := \{y \in \R^d : y = x + tV(x), x \in \Omega\}$ for $t \in [0,1)$.
    Suppose that there is $\epsilon>0$ such that $\Omega_t(V) \in \mathcal{S}$ for all $t \in [0,\epsilon)$.
    We say that $\mathcal{J}\colon \mathcal{S}\to \R$ is (semi-)differentiable at $\Omega$ in direction $V$ if the limit
    \begin{equation}
        \mathcal{J}'(\Omega)[V] := \lim_{t\to 0^+} \frac{\mathcal{J}(\Omega_t(V)) - \mathcal{J}(\Omega)}{t}
    \end{equation}
    exists.
\end{definition}
In many cases, this derivative is seen to be linear and is a Fr\'echet derivative.
For certain gradient based methods, one is interested in the direction of steepest descent.
The notion of a direction of steepest descent requires a topology to be given to the space of variations.

As discussed in the introduction, many articles consider the so-called shape gradient \cite{HP15,HPS15,IglSturWec18,EigStu18}, which uses an application of the Riesz representation theorem in Hilbert spaces.
For a Hilbert space $H$ with inner product $(\cdot,\cdot)_H$, the shape gradient $\nabla_H \mathcal{J} \in H$ is defined by 
\begin{equation}
    (\nabla_H \mathcal{J},v)_H = \mathcal{J}'(\Omega)[v]\quad  \forall v \in H.
\end{equation}
In this case, it is clear that the direction of steepest descent is the negative of the direction of the shape gradient.
The caveat of this approach is that one requires that the shape derivative is bounded on an appropriate Hilbert space.
In general, it is possible to find such a space, however the space may be too small for the problem.

Since the shape derivative is defined for functions in $C^{0,1}$, particularly functions in $\LipO$ (the convex hull of directions in $C^{0,1}$), and this is a closed space, it may be seen as the appropriate space in which to seek directions of steepest descent.
It is of course known that $C^{0,1}$ is not a Hilbert space.

Let us assume that the shape derivative at $\Omega$ in direction $V$ takes the form 
\begin{equation}\label{eq:ShapeDerivativeForm}
    \mathcal{J}'(\Omega)[V] = \int_\Omega S_1 : D V + S_0 \cdot V
\end{equation}
for some $S_1 \in L^1(\Omega;\R^{d\times d})$, $S_0 \in L^1(\Omega;\R^d)$.
The functions $S_1$ and $S_0$ will depend on the solution to the state problem and often on the solution of an adjoint problem.
This assumption is frequently valid and may be verified in applications by calculation.
From this form, one might see that considering functions $V\in W^{1,\infty}(\Omega;\R^d)$ appears more natural than the equivalent choice of functions in $C^{0,1}(\Omega;\R^d)$.
\begin{example}
Let us provide the shape derivative in the setting of the example given in Example \ref{ex:PoissonExample}.
Supposing that $J_\Omega(v) := \int_\Omega j(x,v(x))\dee x$, one may calculate
\begin{equation}
    \mathcal{J}'(\Omega)[V]
    =
    \int_\Omega \mathcal{A}[V] \nabla u \cdot \nabla p + j(\cdot,u) \Div(V) + j_x(\cdot,u) \cdot V + F\nabla p \cdot V,
\end{equation}
where $\mathcal{A}[V]:= I \Div(V) - DV - DV^T$, $j_x$ denotes the derivative of $j$ in its first variable, and $p\in H^1_0(\Omega)$ satisfies the adjoint equation
\begin{equation}
    \int_\Omega \nabla p \cdot \nabla \eta + \eta j_v (\cdot,u) = 0\quad \forall \eta \in H^1_0(\Omega),
\end{equation}
where $j_v$ denotes the derivative of $j$ its second variable.
This results in $S_1$ and $S_0$ being given by
\begin{equation}
    S_1 = \left( I \left(\nabla u \cdot \nabla p\right) - \nabla u \otimes \nabla p - \nabla p \otimes \nabla u \right) + I j(\cdot,u)
\mbox{ and }
    S_0 =  j_x(\cdot,u) + F\nabla p
\end{equation}
where, for vectors $a,b \in \R^d$, $(a \otimes b)_{ij} = a_ib_j$ for $i,j =1,...,d$.
\end{example}

Before we turn to the existence of a direction of steepest descent for the shape derivative, let us develop the above comment further, by comparing the difference between Lipscchitz and $W^{1,\infty}$ functions in this setting.

\subsubsection{A comparison of Lipschitz and $W^{1,\infty}$ functions}
Let us recall from \cite{EvaGar15} that $W^{1,\infty}(\Omega;\R^d) \cong C^{0,1}(\Omega;\R^d)$.
This means that we could replace the occurrence of Lipschitz functions with $W^{1,\infty}$ functions, however,
the semi-norms on the spaces are not equal.
Let us recall the Lipschitz semi-norm: for $V \in C^{0,1}(\Omega;\R^d)$,
\begin{equation}
    |V|_{\Lip(\Omega)} := \sup_{x,y\in \Omega, x\neq y} \frac{|V(x)-V(y)|}{|x-y|}.
\end{equation}
When $\Omega$ is a Lipschitz domain, one has that there is $C_\Omega\geq 1$ such that for all $V \in \Lip(\Omega;\R^d)$
\begin{equation}\label{eq:equivalenceOfSemiNorms}
    \begin{split}
    \|DV\|_{L^\infty(\Omega)} &\leq |V|_{\Lip(\Omega)}\leq C_\Omega \|DV\|_{L^\infty(\Omega)}.
    \end{split}
\end{equation}
The constant $C_\Omega$ will depend on the ratio between the length of the shortest path in $\Omega$ which connects any two points and the distance between the two points.
For convex domains, equality holds, that is $C_\Omega = 1$ and $\|DV\|_{L^\infty(\Omega)} = |V|_{\Lip(\Omega)}$ for all $V \in \Lip(\Omega;\R^d)$.

\begin{example}
Let us give an example whereby these semi-norms seen to be be not equal and how it may cause difficulty in practice.
Consider $d=2$ and the domain $\Omega = B_2(0) \setminus \overline{B_1(0)}$.
For $x = (r\cos(\theta), r\sin(\theta) ) \in \Omega$ with $r \in (0,2)$, $\theta \in (-\pi,\pi]$, let
\begin{equation}
    V(x) = (|\theta|,0),
\end{equation}
then $|DV|\leq 1$ a.e. in $\Omega$.
Fix $\epsilon \in (0,\frac{\pi-2}{2})$, let $a = (1+\epsilon,0)^T$ and $b = -a$.
Observe that $V(a) = 0$ and $V( b) = \pi$, then we see that $ | V(a) - V( b)| = \pi > 2+2\epsilon = |a-b|$.
Furthermore, one may see that $\Phi_t \colon x \mapsto x + tV(x)$ is not invertible for all $t \in (0,1)$, one requires a smaller interval for $t$.
\end{example}
In this example, using the conditions that the $W^{1,\infty}$-semi norm is bounded by one leads to a map which is not necessarily invertible for $t \in (-1,1)$.
This may cause difficulties in algorithms which update the domain using large steps e.g. an Armijo line search rule.
Of course one may choose a smaller maximum step.
The maximum interval over which $\Phi_t$ is invertible will depend on the domain, the size of this interval can be difficult to compute in applications.
Furthermore, let us mention that the $W^{1,\infty}$ semi-norm loses a lot of non-local information that is known to the Lipschitz semi-norm.
In applications this could possibly lead to a false minimising sequence.

\subsection{Seeking a direction of steepest descent for the shape derivative}
The problem of finding a direction of steepest descent is posed as follows:
\begin{problem}\label{prob:MinimisingDirection}
    Given $\Omega \in \mathcal{S}$, find $V\in \LipO (\Omega;\R^d)$ such that
    \begin{equation}
        \mathcal{J}'(\Omega)[V] = \inf \left\{\mathcal{J}'(\Omega)[V'] : V' \in \LipO (\Omega;\R^d) \right\}.
    \end{equation}
\end{problem}
Notice that this problem need not be well-posed.
Ill-posedness can arise from the fact that $\mathcal{J}'(\Omega)$ need not be bounded below on $\LipO(\Omega;\R^d)$.
The lack of bound below may appear since we control the $C^{0,1}(\Omega;\R^d)$ semi-norm, rather than the full norm.

The following example demonstrates this ill-posedness:
\begin{example}\label{ex:ConstraintNeededForMinimisation}
Suppose that $\int_\Omega S_0 = (1,0)^T$.
Consider the sequence of constant vector fields
\begin{equation}
    V_n := - n(1,0)^T \mbox{ for } n\geq 1.
\end{equation}
It is clear that $V_n \in \LipO(\Omega;\R^d)$ for $n \geq 1$ and that $\mathcal{J}'(\Omega)[V_n] = -n \to -\infty$, therefore $\mathcal{J}'(\Omega)$ is not bounded below on $\LipO(\Omega;\R^d)$.
\end{example}

One may resolve this ill-posedness by restricting the space in which minimisers are sought over.
Let us also mention that one could also use a full norm on $C^{0,1}$, for example adding on $L^\infty$ norm or the evaluation at a single point to the semi-norm.
In applications, it is frequent to have a restriction which will allow for the well-posdeness.
\begin{example}
Let us now mention two constraints which arise in applications and will allow for the well-posedness.
\begin{itemize}
    \item
    In problems involving elasticity it is common to consider that part of the boundary is fixed.
    To ensure that $\Phi_t$ fixes parts of the boundary, say $\Sigma \subset \partial \Omega$, one should restrict the perturbation $V$ to satisfy $V|_{\Sigma} =0$.
    \item
    It might be necessary to require that the domain which is being optimised must have a fixed centre of mass.
    To constrain $\Phi_t (\Omega)$ to satisfy $\int_{\Phi_t(\Omega)} \id =0$, the first order condition on perturbations $V$ is given by $\int_\Omega V + \id \Div V =0$.
\end{itemize}
\end{example}
Both of these constraints lead to a Poincar\'e-type inequality.
It is is worth mentioning that it is also common to consider transformations which preserve the volume.

Let us now state the Poincar\'e type results we mentioned above.
\subsubsection{Poincar\'e inequalities}\label{subsubsec:Poincares}
We begin with the Poincar\'e inequality with a clamped condition, the proof of which is fairly standard.
\begin{lemma}\label{lem:PoincareClamped}
    Let $\Sigma \subset \partial \Omega$.
    There is $C>0$ such that
    \begin{equation}
        \|V\|_{L^\infty(\Omega;\R^d)} \leq C \|DV\|_{L^\infty(\Omega;\R^{d\times d})}
    \end{equation}
    for all $V \in W^{1,\infty}(\Omega;\R^d)$ such that $V|_{\Sigma}=0$.
    The constant $C$ is related to the maximum intrinsic distance in $\Omega$ between points in $\Sigma$ and $\Omega$.
\end{lemma}
It is not immediately clear how to directly prove such a Poincar\'e inequality over the space of functions which satisfy $\int_\Omega V + \id \Div V =0$.
Let us quote the following result which may be found in \cite{AleMorRos08} as Lemma 2.1.
\begin{lemma}\label{lem:Poincares}
    Let $X,Y$ be Banach spaces and $L\colon X \to Y $ a bounded linear operator and $X_0 \subset X$ be its null-space.
    Let $P \colon X \to X_0$ be a bounded linear operator such that $P|_{X_0} = \id$.
    Assume that there is $K>0$ such that
    \begin{equation}
        \|x- Px\|_X \leq K\|L x\|_Y \quad \forall x \in X.
    \end{equation}
    For $T \colon X \to X_0$ a bounded linear operator such that $T|_{X_0} = \id$, it holds that
    \begin{equation}
        \|x-Tx\|_X \leq (1+ \|T\|_{op}) K \|Lx\|_Y\quad \forall x \in X.
    \end{equation}
\end{lemma}
This result effectively states that we can use a known Poincar\'e inequality to produce the inequality we desire.
\begin{lemma}\label{lem:PoincareAverageValue}
    There is $C>0$ such that
    \begin{equation}
        \|V\|_{L^\infty(\Omega;\R^d)} \leq C \|DV\|_{L^\infty(\Omega;\R^{d\times d})}
    \end{equation}
    for all $V \in W^{1,\infty}(\Omega;\R^d)$ such that $\int_\Omega V =0$.
    The constant $C$ is related to the maximum intrinsic distance in $\Omega$ between any two points of $\Omega$.
\end{lemma}
\begin{lemma}\label{lem:PoincareBarycenter}
    There is $C>0$ such that
    \begin{equation}
        \|V\|_{L^\infty(\Omega;\R^d)} \leq C \|DV\|_{L^\infty(\Omega;\R^{d\times d})}
    \end{equation}
    for all $V \in W^{1,\infty}(\Omega;\R^d)$ such that $\int_\Omega V + \id \Div V =0$.
\end{lemma}
\begin{proof}
    To prove this, we apply Lemma \ref{lem:Poincares} with $X = W^{1,\infty}(\Omega;\R^d)$, $Y = L^\infty(\Omega;\R^{d\times d})$, $L \colon V \mapsto DV$, $X_0$ as the set of constant functions, $P \colon V \mapsto \frac{1}{|\Omega|}\int_\Omega V$ and $T \colon V \mapsto \frac{1}{|\Omega|}\int_\Omega V + \id \Div V$.
    In Lemma \ref{lem:PoincareAverageValue}, it is stated that $K$ is related to the maximum intrinsic distance in $\Omega$ between any two points of $\Omega$.
\end{proof}

\subsubsection{Existence of a direction of steepest descent under constraints}

In the case that one has a Poincar\'e inequality, existence of a solution follows from the arguments presented in \cite[Proposition 3.1]{PaWeFa18} or \cite[Theorem 2.5]{DecHerHin21}.
The proof is effectively an application of the Direct Method of the Calculus of Variations.
\begin{proposition}\label{prop:ThereIsASolution}
    Let $U\subset \LipO(\Omega;\R^d)$ be either:
    \begin{equation}
    \begin{split}
        \left\{V' \in \LipO(\Omega;\R^d) : \int_{ \Omega} V' + \id \Div V = 0\right\}
        \mbox{ or }
        \left\{V' \in \LipO(\Omega;\R^d) : V'|_{\Sigma} =0 \right\},
    \end{split}
    \end{equation}
    where $\Sigma \subset \partial \Omega$.
    There exists $V \in U$ such that
    \begin{equation}\label{eq:ThereIsAsSolution}
    \mathcal{J}'(\Omega)[V] = \inf \{\mathcal{J}'(\Omega)[V'] : V' \in U \}.
    \end{equation}
\end{proposition}
\begin{proof}
    Let $\{V_n\}_{n=1}^\infty \subset U$ be a infimising sequence.
    By either of the Poincar\'e inequalities in Lemmas \ref{lem:PoincareClamped} or \ref{lem:PoincareBarycenter}, it holds that for all $n \geq 1$
    \begin{equation}
        \|V_n\|_{W^{1,\infty}(\Omega;\R^d)}
        =
        \|V_n\|_{L^\infty(\Omega;\R^d)}
        +
        \|DV_n\|_{L^\infty(\Omega;\R^{d\times d})}
        \leq
        C + 1,
    \end{equation}
    where $C$ is the appropriate Poincar\'e constant.
    It is known that bounded sequences are weak-$*$ compact, therefore there is a subsequence (which we do not relabel) and target $V^* \in U$ such that $V_n \overset{*}{\rightharpoonup} V^* $ in $W^{1,\infty}(\Omega;\R^d)\cong C^{0,1}(\Omega;\R^d)$, where it is known that $V^* \in U$ since the constraints in the definition of $U$ are weak-$*$ continuous.
    As we have assumed that $\mathcal{J}'(\Omega)[V] = \int_\Omega S_1 : DV + S_0 \cdot V$ for $V \in W^{1,\infty}(\Omega;\R^d)$, it holds that $\mathcal{J}'(\Omega)$ is weak-$*$ continuous, as such
    \begin{equation}
        \mathcal{J}'(\Omega)[V_n] \to \mathcal{J}'(\Omega)[V^*] = \inf \{\mathcal{J}'(\Omega)[V'] : V' \in U \}\mbox{ as }n \to \infty.
    \end{equation}
\end{proof}
The choices of $U$ in the above proposition relates to choosing perturbations which fix the centre of mass at the origin (to first order), or fix a part of the boundary, as we have discussed before.
We note that this proposition does not give uniqueness of minimising directions nor a clear strategy of how to approximate them.
It is worth mentioning that, when $\Omega$ is convex, one may approximate a minimiser by considering a relaxation using the $p$-Laplacian, see \cite[Remark 2.6]{DecHerHin21}.

\begin{remark}\label{rem:HoldAllDomainForLipschitzW1infty}
    The Lipschitz bound in the condition that $V \in \LipO(\Omega;\R^d)$ is, for non-convex domains, a non-local condition, this can make a computational implementation rather difficult.
    It is possible to include a fictitious domain $D$, a so-called hold-all domain, which is bounded and convex, with $\Omega \Subset D$.
    On this fictitious domain, as we previously discussed, the Lipschitz semi-norm is equal to the $W^{1,\infty}$ semi-norm.
\end{remark}

See that, if $S_1 \in BV(\Omega;\R^{d\times d})$, \eqref{eq:ShapeDerivativeForm} may be rewritten in the form of
\begin{equation}\label{eq:ShapeOptimisationMeasure}
    \mathcal{J}'(\Omega)[V]
    =
    \int V \cdot \dee j,
\end{equation}
where $\dee j =  S_0 \dee \mathcal{L}^d\measurerestr \Omega + S_1 \normal\ \dee \mathcal{H}^{d-1}\measurerestr {\partial \Omega} -\dee ( \Div S_1 ) $ is a vector measure, where $\mathcal{L}^d$ the $d$-dimensional Lebesgue measure, $\mathcal{H}^{d-1}$ the $(d-1)$-dimensional Hausdorff measure and $\normal$ is the outward unit normal on $\partial\Omega$.
Later it will be seen that minimising an energy of the form of \eqref{eq:ShapeOptimisationMeasure} has a very similar structure to that of the dual problem in Optimal Transport.
We will later describe the setting of star-shaped domains which, with a constraint on the area, will allow for a link to Optimal Transport to be utilised to approximate directions of steepest descent.

\subsection{Continuous approximations to steepest descents}
Let us take a moment to elaborate on the comment after Proposition \ref{prop:ThereIsASolution}.
We first state that a there exist a solution to a $p$-relaxed problem.
Such a relaxation appears in \cite{IshLor05} with a generalisation in \cite{LopNavRos13}.
The approximation we analyse has been used in the works \cite{MulKulSie21,MulPinRun22}.
\begin{proposition}\label{prop:ThereIsARelaxedSolution}
    Let $p > d$ and $U_p\subset W^{1,p}(\Omega;\R^d)$ be either:
    \begin{equation}
    \begin{split}
        \left\{V' \in W^{1,p}(\Omega;\R^d) : \int_{ \Omega} V' + \id \Div V' = 0\right\}
        \mbox{ or }
        \left\{V' \in W^{1,p}(\Omega;\R^d) : V'|_{\Sigma} =0 \right\},
    \end{split}
    \end{equation}
    where $\Sigma \subset \partial \Omega$ has positive $\mathcal{H}^{d-1}$ measure.
    Assume that $S_1 \in L^{p'}(\Omega;\R^{d\times d})$ where $\frac{1}{p}+\frac{1}{p'}=1$.
    There exists $V_p \in U_p$ such that
    \begin{equation}
        \mathcal{J}'(\Omega)[V_p] + \frac{1}{p} \int_{\Omega}|DV_p|^p
        =
        \inf \left\{ \mathcal{J}'(\Omega)[V'] + \frac{1}{p}\int_\Omega |DV'|^p : V' \in U_p \right\}.
    \end{equation}
\end{proposition}
This proof is again an application of the Direct Method of the Calculus of Variations and relies on Poincar\'e inequalities on $U_p$, which may be derived in a similar way to those which appear in Section \ref{subsubsec:Poincares}.
With the existence of a solution to the $p$-relaxed problem, let us now mention the limit as $p \to \infty$.
\begin{proposition}\label{prop:InfinityProblem}
    For $p>d$ let $V_{p} \in U_{p}$ be as in Proposition \ref{prop:ThereIsARelaxedSolution}.
    Assume that there is $p_*>d$ such that $S_1 \in L^{p_*'}(\Omega;\R^{d\times d})$.
    There is a sequence $p_j \to \infty$ and function $V_\infty \in W^{1,\infty}(\Omega;\R^d)$ such that the sequence $V_{p_j}$ converges weakly to $V_\infty$ in $U_{p}$ for any $p>p_*$, furthermore it holds that $|DV_\infty| \leq 1$ a.e. .
\end{proposition}
\begin{proof}
    Given $p >d$, let $q\in (d,p)$, then it holds that
    \begin{equation}
        \|DV_p\|_{L^q(\Omega;\R^{d\times d})}
        \leq
        |\Omega|^{\frac{1}{q}-\frac{1}{p}} \|DV_p\|_{L^p(\Omega;\R^{d\times d})}.
    \end{equation}
    while it also holds that for $p> p_*$
    \begin{equation}
        \begin{split}
            \|DV_p\|_{L^p(\Omega;\R^{d\times d})}^p
            \leq
            \|\mathcal{J}'(\Omega) \|_{\left( W^{1,p}(\Omega;\R^d) \right)^*}
            \left( 1+ C_p \right)\|DV_p\|_{L^p(\Omega;\R^{d\times d})}
        \end{split}
    \end{equation}
    where for the existence of $V_p$ we have assumed that $\mathcal{J}'(\Omega)$ is bounded on $W^{1,p}(\Omega;\R^d)$ (by virtue of assuming $S_1 \in L^{p_*'}(\Omega;\R^{d\times d})$ ).
    It holds that as $p \to \infty$, $C_p$ is bounded.
    Both of these estimates demonstrate that
    \begin{equation}
        \|DV_p\|_{L^q(\Omega;\R^{d\times d})}
        \leq
        |\Omega|^{\frac{1}{q}-\frac{1}{p}} \left( \|\mathcal{J}'(\Omega) \|_{\left( W^{1,p}(\Omega;\R^d) \right)^*}
            \left( 1+ C_p \right) \right)^{\frac{1}{p-1}},
    \end{equation}
    from which, one may see that $\{V_p\}_{p > p_*}$ is bounded in $W^{1,q}(\Omega;\R^d)$, therefore there is a function $V_\infty$ and sequence $p_j\to \infty$ such that $V_{p_j} \rightharpoonup V_\infty$ in $W^{1,q}(\Omega;\R^d)$ and by Sobolev embedding, we know that it converges strongly as a continuous function.
    By weak-lower-semi-continuity of norms, it holds that for all $q> p_*$,
    \begin{equation}
        \|DV_\infty\|_{L^p(\Omega;\R^{d\times d})} \leq |\Omega|^{\frac{1}{q}},
    \end{equation}
    which implies that $|DV_\infty|\leq 1$ almost everywhere.
\end{proof}
We have considered that $S_1$ is slightly more integrable than $L^1$, this is not necessary and with very little difference to the above, one could indeed consider an approximation of $S_1$, say $S_1'$, which converges to $S_1$ in $L^1$.
Notice that $V_\infty$ produced in Proposition \ref{prop:InfinityProblem} may not lie in $\LipO(\Omega;\R^d)$ when $\Omega$ is not convex.
By recalling the inequality \eqref{eq:equivalenceOfSemiNorms}, it holds that $\mathcal{J}'(\Omega) [V_\infty] \leq \mathcal{J}'(\Omega) [V]$ where $V \in \LipO(\Omega;\R^d)$ is as in Proposition \ref{prop:ThereIsASolution}.
While we know that $V_p$ has subsequences which converge to $V_\infty$, one may estimate how close the descents they generate are.

\begin{proposition}\label{prop:pApprox}
    Let $p>d$, let $V_\infty$ be as in Proposition \ref{prop:InfinityProblem}, let $V_p \in U_p$ be as in Proposition \ref{prop:ThereIsARelaxedSolution}.
    It holds that
    \begin{equation}
        \mathcal{J}'(\Omega)[V_p] \leq \mathcal{J}'(\Omega)[V_\infty] + \left(\frac{p}{p-1}\right)^{p-1}\frac{|\Omega|}{p}.
    \end{equation}
\end{proposition}
\begin{proof}
    By considering the Euler-Lagrange equation for $V_p$, we see that
    \begin{equation}
        \int_\Omega |DV_p|^{p-2} DV_p : D V' + \mathcal{J}'(\Omega)[V'] = 0 \quad \forall V' \in U_p.
    \end{equation}
    Taking $V' = V_p -V_\infty$, one has that
    \begin{equation}
        \begin{split}
        \mathcal{J}'(\Omega)[V_p - V_\infty]
        =&
        \int_\Omega |DV_p|^{p-2} DV_p : D (V_\infty-V_p)
        \\
        \leq&
        \int_\Omega |DV_p|^{p-1} - \int_\Omega |DV_p|^p,
        \end{split}
    \end{equation}
    where we have used that $|DV_\infty|\leq 1$.
    We now estimate with  Young's inequality
    \begin{equation}
        \int_\Omega |DV_p|^{p-1}
        \leq \int_\Omega \left(|DV_p|^p + \frac{1}{p} \left( \frac{p}{p-1}\right)^{p-1}\right)
    \end{equation}
    which completes the result.
\end{proof}
While this problem is not non-smooth, like that posed in \eqref{eq:ThereIsAsSolution}, it has a degenerate elliptic operator.
In the case that $\Omega$ is not convex, let us again mention that using a convex hold-all domain may be of help.
Recall that a hold-all domain was mentioned in Remark \ref{rem:HoldAllDomainForLipschitzW1infty} as a method in which to allow use of the local constraint of $|DV|\leq 1$ a.e. to handle the non-local Lipschitz constraint.
This use of a hold-all domain will be explored in upcoming work.

Let us now mention a viscosity-type approach to approximate a solution.
\begin{proposition}\label{prop:ThereIsAViscousSolution}
    Let $\epsilon >0$ and $U \subset \LipO(\Omega;\R^d)$ be either:
    \begin{equation}
    \begin{split}
        \left\{V' \in \LipO(\Omega;\R^d) : \int_{\partial \Omega} V' + \id \Div V' = 0\right\}
        \mbox{ or }
        \left\{V' \in \LipO(\Omega;\R^d) : V'|_{\Sigma} =0 \right\},
    \end{split}
    \end{equation}
    where $\Sigma \subset \partial \Omega$.
    There exists $V^\epsilon \in U$ such that
    \begin{equation}
    \mathcal{J}'(\Omega)[V^\epsilon] + \frac{\epsilon}{2} \int_\Omega |D V^\epsilon|^2
    =
    \inf \left \{ \mathcal{J}'(\Omega)[V'] + \frac{\epsilon}{2} \int_\Omega |D V'|^2 : V' \in U \right\}.
    \end{equation}
\end{proposition}
The existence is again given by an application of the Direct Method of the Calculus of Variations.
We see that, although the problem is not non-linear in the way of the $p$-relaxation, it instead has a convex constraint.
We may also show that $V^\epsilon$ achieves close to a steepest descent.
\begin{proposition}\label{prop:ViscousApprox}
    Let $\epsilon$, let $V \in U$ be as in Proposition \ref{prop:ThereIsASolution}, let $V^\epsilon \in U$ be as in Proposition \ref{prop:ThereIsAViscousSolution}.
    It holds that
    \begin{equation}
        \mathcal{J}'(\Omega)[V] \leq \mathcal{J}'(\Omega)[V^\epsilon] \leq \mathcal{J}'(\Omega)[V] + \frac{\epsilon}{4}|\Omega|.
    \end{equation}
\end{proposition}
\begin{proof}
    The first inequality follows immediately from the fact that $V$ is minimal over $U$.
    The proof of the second inequality follows similarly to the proof of Proposition \ref{prop:pApprox}.
    By considering the Variational characterisation of $V^\epsilon$, that is
    \begin{equation}
        \mathcal{J}'(\Omega)[V^\epsilon] + \epsilon\int_\Omega DV^\epsilon : D V^\epsilon
        \leq
        \mathcal{J}'(\Omega)[V'] + \epsilon\int_\Omega DV^\epsilon : D V'
        \quad \forall V' \in U.
    \end{equation}
    Taking $V' = V$, one has that
    \begin{equation}
        \begin{split}
        \mathcal{J}'(\Omega)[V^\epsilon - V]
        =&
        \epsilon \int_\Omega DV^\epsilon : D (V-V^\epsilon)
        \\
        \leq&
        \epsilon \int_\Omega |DV^\epsilon | - |DV^\epsilon|^2
        \leq
        \frac{ \epsilon}{4}|\Omega|
        ,
        \end{split}
    \end{equation}
    where we have used that$|DV|\leq 1$ a.e. since $V\in \LipO(\Omega;\R^d)$.
\end{proof}

\section{Introduction to Optimal Transport}\label{sec:OptTran}
We now give a brief introduction to Optimal Transport so that we may see how the problem given in \eqref{eq:ShapeOptimisationMeasure} may be related.
We refer to \cite{San15} for any details which we omit in this section.
The first problem of Optimal Transport is:
\begin{problem}\label{prob:FirstOT}
Given probability measures $\mu\in \mathcal{P}(X)$, $\nu\in\mathcal{P}( Y)$ and non-negative proper function $c\colon X \times Y \to \R$, find $T\colon X \to Y$ such that $\int_X c(x,T(x)) {\rm d} \mu(x)$ is minimised such that the push forward of $\mu$ under $T$ is $\nu$, that is $T_\# \mu = \nu$.
\end{problem}
It is known that this problem need not admit a solution which is a function, one therefore extends the notion of solution to probability measures on $X\times Y$.
The following relaxation is known as the Kantorovich problem:
\begin{problem}\label{prob:KantorovichProb}
Given $\mu \in \mathcal{P}(X)$, $\nu \in \mathcal{P}(Y)$ and $c\colon X \times Y\to \R$ a non-negative proper function, find $\gamma \in\Pi(\mu,\nu)$ such that
\begin{equation}\label{eq:KantorovichMinimisation}
    \int_{X\times Y} c \dee \gamma = \inf_{\gamma' \in \Pi(\mu,\nu)} \int_{X\times Y} c \dee \gamma',
\end{equation}
where $\Pi(\mu,\nu):= \left\{\tilde \gamma\in \mathcal{P}(X\times Y) : (\pi_x)_\# \tilde\gamma = \mu,\, (\pi_y)_\# \tilde \gamma = \nu \right\}$, where $\pi_x$ (resp. $\pi_y$) is the projection from $X\times Y$ onto $X$ (resp. $Y$).
\end{problem}

For the convenience of the reader, let us mention how one may understand the role of $\gamma$ as a generalisation of a map $T$.
For $A \subset X$ and $B\subset Y$ measurable, the value $\gamma(A\times B)$ gives the amount of mass which is transported from $A$ to $B$.

It is well documented that when $X$ and $Y$ are compact metric spaces and $c$ is continuous, Problem \ref{prob:KantorovichProb} has a solution, which follows by an application of the Direct Method of the Calculus of Variations, where the topology considered is that of the weak convergence of Probability measures.

In the study of convex minimisation problems, an important tool is duality.
For us, a particular form of the dual problem will be the one which relates to Shape Optimisation.
\begin{problem}\label{prob:FirstDualProblem}
Let $\mu \in \mathcal{P}(X)$, $\nu \in \mathcal{P}(Y)$ and $c\colon X \times Y\to \R$ a bounded non-negative proper function.
Find bounded continuous functions $\phi\in C_b(X)$, $\psi \in C_b(Y)$ such that $\phi(x) + \psi(y) \leq c(x,y)$ for all $(x,y) \in X\times Y$ and
\begin{equation}
    \int_X \phi \dee \mu + \int_Y \psi \dee \nu
    =
    \sup\left\{\int_X \phi' \dee \mu + \int_Y \psi' \dee \nu :
    \begin{split}
        &\phi'\in C_b(X),\, \psi' \in C_b(Y),
        \\
        &\phi'(x) + \psi'(y) \leq c(x,y) \ \forall (x,y) \in X\times Y
    \end{split}\right\}.
\end{equation}
\end{problem}
There exists a solution to this problem, see \cite[Proposition 1.11]{San15}.
For $\gamma$ an admissible measure in Problem \ref{prob:KantorovichProb} and $(\phi,\psi)$ an admissible pair in Problem \ref{prob:FirstDualProblem}, by integrating the constraint $\phi(x) + \psi(y) \leq c(x,y)$ $\forall (x,y) \in X\times Y$ against $\gamma$, it is seen that
\begin{equation}
    \int_X \phi \dee \mu + \int_Y \psi \dee \nu \leq \int_{X\times Y} c \dee \gamma.
\end{equation}

It is possible to relax Problem \ref{prob:FirstDualProblem} by the so-called $c$ transform.
Given $\chi \colon X \to \R$, define $\chi^c\colon Y \to \R$ by
\begin{equation}
    \chi^c(y):= \inf_{x \in X} c(x,y)-\chi(x).
\end{equation}
Analogously, given $\xi \colon Y \to \R$, define $\xi^{\bar c} \colon X \to \R$ by
\begin{equation}
    \xi^{\bar c} (x):= \inf_{y \in Y} c(x,y)-\xi(y).
\end{equation}
With this, one says a function $\phi \in C_b(X)$ is $c$-concave if there exists $\chi\colon Y \to \R$ such that $\phi = \chi^{\bar c}$.
One has an analogous definition for $\bar{c}$-concave functions.

One may wish to pose a restricted problem, which is the maximisation over $c$-concave functions, rather than over pairs of functions.
\begin{problem}\label{prob:SecondDualProblem}
Given $\mu \in \mathcal{P}(X)$, $\nu \in \mathcal{P}(Y)$ and $c\colon X \times Y\to \R$ a bounded non-negative proper function, find $c$-concave function $\phi\in C_b(X)$ such that
\begin{equation}
    \int_X \phi \dee \mu + \int_Y \phi^c \dee \nu = \sup\left\{\int_X \phi' \dee \mu + \int_Y (\phi')^c \dee \nu :\phi'\in C_b(X) \mbox{ is $c$-concave}\right\}.
\end{equation}
\end{problem}
The fact that Problems \ref{prob:FirstDualProblem} and \ref{prob:SecondDualProblem} are equivalent when $X$ and $Y$ are compact may be found in \cite[Proposition 1.11]{San15}.

In this work, we are particularly interested in the case that $X = Y$ and $c$ is a metric, in which case, one has the result \cite[Proposition 3.1]{San15} which states that a function is $c$-concave if and only if it is Lipschitz (with respect to the metric $c$) with Lipschitz bound less than or equal to $1$, furthermore, $\phi^c = - \phi$.
In the case that $c$ is a metric, we write
\begin{equation}
    \LipO(X) := \left\{ \phi \in C_b(X) : |\phi(x) - \phi(y)| \leq c(x,y) \ \forall x,y \in X\right\}.
\end{equation}
A reformulation of Problem \ref{prob:SecondDualProblem} leads to the following problem.
\begin{problem}\label{prob:OTThatWeWant}
Given $\mu,\nu \in \mathcal{P}(X)$ and $c\colon X \times X\to \R$ is a metric, find $\phi\in \LipO(X)$ such that 
\begin{equation}
    \int_X \phi \dee \left(\mu- \nu\right) = \sup\left\{\int_X \phi' \dee\left( \mu-\nu\right) : \phi' \in \LipO(X)\right\}.
\end{equation}
\end{problem}
Together, we have the following duality result.
\begin{theorem}\label{thm:DualityResult}
Let $X$ be a compact metric space with metric $c\colon X \times X \to \R$, let $\mu,\nu \in \mathcal{P}(X)$, then
\begin{equation}
    \min\left\{\int_{X\times X} c\dee \gamma : \gamma \in \Pi(\mu,\nu) \right\}
    =
    \max\left\{ \int_X \phi \dee (\mu - \nu) : \phi \in \LipO(X) \right\}.
\end{equation}
\end{theorem}

We see that the problem of minimising $\phi\mapsto \int_X \phi \dee (\mu-\nu)$ over $\Lip_1(X)$ is almost \emph{a scalar version} of the problem of minimising $V \mapsto \int_\Omega V \cdot \dee j$ over $\LipO(\Omega;\R^d)$ which was mentioned in \eqref{eq:ShapeOptimisationMeasure}.

\subsection{Approximations for Optimal Transport}\label{sec:sinkhorn}
Let us mention that one may approximate the Optimal Transport problems by entropic regularisation.
For the continuous case of entropic regularisation, we refer to \cite{ClaLorMah21}.
In the discrete setting, this leads to the so-called Sinkhorn algorithm \cite{Cut13,Kni08,BenCarGui15}.
Since we will use the Sinkhorn algorithm, let us briefly outline the algorithm.
Before this, we give necessary definitions for the discrete problem of Optimal Transport.

Let $a\in \R^{n_1}$, $b \in \R^{n_2}$ such that $a,b\geq 0$ and $\sum_{i=1}^{n_1}a_i = \sum_{j=1}^{n_2} b_j$, let us suppose they are equal to $1$.
Here $a$ may be identified with $\mu$ and $b$ with $\nu$.
Let us set
\begin{equation}
    U(a,b) := \left \{ P \in \R^{n_1\times n_2},\ P\geq 0,\ \sum_{j=1}^{n_2}P_{ij} = a_i,\ \sum_{i=1}^{n_1}P_{ij} = b_j \right\}
\end{equation}
which is analogous to $\Pi(\mu,\nu)$.
Let $C\in \R^{n_1\times n_2}$ be a cost matrix, the discrete problem of Optimal Transport is to find
\begin{equation}
    P^* \in \argmin \left \{ \sum_{i=1}^{n_1}\sum_{j=1}^{n_2} P_{ij} C_{ij} : P \in U(a,b) \right\}.
\end{equation}
Letting $\epsilon >0$ to be a regularisation parameter, the regularised problem is to find
\begin{equation}
    P^*_\epsilon \in \argmin \left \{ \sum_{i=1}^{n_1}\sum_{j=1}^{n_2} P_{ij} C_{ij} + \epsilon P_{ij}\left(\log(P_{ij}) -1\right)  : P \in U(a,b) \right\}.
\end{equation}
Set $K\in \R^{n_1\times n_2}$ to be given component-wise by $K_{ij} = \exp(-\frac{1}{\epsilon}C_{ij})$.
The Sinkhorn algorithm is given as:
let $u^0 \in \R^{n_1}$ with $u^0_i = 1$ and $v^0 \in \R^{n_2}$ with $v^0_j = 1$, then update
\begin{equation}
    u_i^{l+1} = \frac{a_i}{(Kv^l)_i},\ i =1,...,n_1,
    \quad
    v_j^{l+1} = \frac{b_j}{(K^Tu^{l+1})_j},\ j =1,...,n_2
\end{equation}
for $l \geq 0$.

We now discuss star-shaped domains and their optimisation, which will allow us to solidify a link between Shape Optimisation and Optimal Transport.

\section{The optimisation of star-shaped domains}\label{sec:StarShaped}
We now discuss the problems of Shape Optimisation in a star-shaped domain.
The restriction to a star-shaped domain appears in \cite{EppHarRei07,BouChSa20}, where such a restriction allows for a deeper analysis.
A stronger simplification would be to restrict to convex domains, this appears in \cite{BarWac20,KelBarWac21}.
A model Laplace problem was tackled in this star-shaped setting by \cite{DecHerHin21}, whereby a link to Optimal Transport was exploited in the numerical experiments, making use of a Sinkhorn algorithm to find a direction of descent.
We will now introduce some of the relevant concepts, for more details, we refer the reader to the source.

For star-shaped domains, one may completely describe the domain by a point and a so-called radial function.
Given $\Omega$, a bounded star-shaped domain with a centre at $0$, we define $f_\Omega \colon \mathbb{S}^{d-1}\to \R$ by
\begin{equation}
    f_\Omega(\omega) := \sup \{ s\in \R : s\omega \in \Omega \}, \ \omega \in \mathbb{S}^{d-1}.
\end{equation}

Using this radial function it is known that there is a one to one correspondence between star-shaped bounded Lipschitz domains which contain $0$ and strictly positive Lipschitz functions, see
\cite[Section 3.2, Lemma 2]{Bur98} and \cite[Lemma 2.1]{DecHerHin21}.
While equivalent, we choose to equip $\mathbb{S}^{d-1}$ with the intrinsic metric so that the Lipschitz semi-norm (with this intrinsic metric) is equal to the $W^{1,\infty}$ semi-norm.
Let us denote this intrinsic metric by $d$.

Given $f\colon \mathbb{S}^{d-1}\to \R$ with $f> 0$, we set
\begin{equation}
    \Omega_f := \left\{ x \in \R^d : x=0 \mbox{ or } |x| < f(\omega_x) \mbox{ for } x \neq 0\right\} \mbox{ where }\omega_x := \frac{x}{|x|},\ x \neq 0
\end{equation}
and for convenience
\begin{equation}
    J(f) := \mathcal{J}(\Omega_f).
\end{equation}

With this link between strictly positive functions and bounded star-shaped Lipschitz domains in mind, we consider a reference domain, the unit ball, and a map $\Phi_f$ which will take the unit ball to the domain $\Omega_f$.
For $f\in C^{0,1}(\mathbb{S}^{d-1})$ strictly positive, define $\Phi_f \colon B:= \{x \in \R^d : |x|<1\} \to \R^d$ by
\begin{equation}
    \Phi_f(x):= \begin{cases}
    f(\omega_x)x, &x\neq 0,\\
    0, &x=0.
    \end{cases}
\end{equation}
As shown in \cite[Lemma 2.1]{DecHerHin21}, this function is bi-Lipschitz onto its image.

By virtue of the domain being star-shaped and wishing to preserve this structure, it is clear that one need not consider the shape derivative in the direction of general vector-valued perturbations.
It is sufficient to restrict to perturbations $V\in C^{0,1}(\Omega_f;\R^d)$ which take the form
\begin{equation}\label{eq:starShapedVariation}
    V_g(y) =
    \begin{cases}
    \frac{g(\omega_y)}{f(\omega_y)} y, &y \neq 0,\\
    0, &y=0,
    \end{cases}
\end{equation}
for some $g \in C^{0,1}(\mathbb{S}^{d-1})$.
For $t$ such that $f+tg >0$, it holds that $(\id + t V_g)(\Omega_f) = \Omega_{f+tg}$.
With this specific choice of $V$, it is useful to define the derivative of $J$ by
\begin{equation}
    \langle J'(f),g\rangle := \mathcal{J}'(\Omega_f)[V_g].
\end{equation}

As previously mentioned, it is often relevant to incorporate constraints into the direction of steepest descent for the Shape Optimisation problem.
\begin{example}
We give three common constraints and how they may be incorporated into this star-shaped setting and the particular choice of $V$ in \eqref{eq:starShapedVariation}.
\begin{itemize}
    \item
    In order to clamp part of the boundary, say $\Sigma_f \subset \partial\Omega_f$, one should restrict to directions $g$ which satisfy
    \begin{equation}
        g =0 \mbox{ on }\Phi^{-1}_f(\Sigma_f).
    \end{equation}
    \item
    In order to fix the centre of mass at the origin to first order, one should restrict the direction $g$ to satisfy
    \begin{equation}
        \int_{\Omega_f} \left( y \Div\left( \frac{g(\omega_y)}{f(\omega_y)} y \right) + \frac{g(\omega_y)}{f(\omega_y)} y \right) \dee y =0.
    \end{equation}
    \item
    One may wish to fix the volume of $\Omega_f$.
    In order to do this to first order, one should restrict perturbations $g$ to satisfy
    \begin{equation}
        \int_{\Omega_f} \Div \left( \frac{g(\omega_y)}{f(\omega_y)} y \right) \dee y =0.
    \end{equation}
    Through an integration by parts, the functional may also be given by
    \begin{equation}
        \int_{\Omega_f} \Div \left( \frac{g(\omega_y)}{f(\omega_y)} y \right) \dee y
        =
        \int_{\partial\Omega_f} \frac{g(\omega_y)}{f(\omega_y)} \normal \cdot y \dee \mathcal{H}^{d-1}(y).
    \end{equation}
    The following representation, which is given in \cite{DecHerHin21}, is seen to be convenient
    \begin{equation}
        \int_{\Omega_f} \Div \left( \frac{g(\omega_y)}{f(\omega_y)} y \right) \dee y
        =
        \int_{\mathbb{S}^{d-1}} g(\omega) f^{d-1}(\omega) \dee \omega.
    \end{equation}
\end{itemize}
\end{example}

\subsection{Finding the direction of steepest descent and the problem of Optimal Transport}
The problem to minimise $g \mapsto \langle J'(f),g \rangle$ over $g \in \LipO(\mathbb{S}^{d-1})$ is now a problem for a scalar function, which one may see is a lot closer to Problem \ref{prob:OTThatWeWant}.
It is not clear whether the mass is balanced, i.e. whether $\langle J'(f), 1\rangle = \int_{\Omega_f} \frac{y}{f(\omega_y)}\cdot \dee (S_0 + S_1 \normal \mathcal{H}^{d-1} - (\Div S_1) (y)$ vanishes.
If the mass is not balanced then we see that a minimiser may not exist - this may be demonstrated in a very similar way as to the vector case in Example \ref{ex:ConstraintNeededForMinimisation}.
A lack of balancing takes it away from the traditional setting of Optimal Transport.

If one considers the problem of finding 
\begin{equation}\label{eq:findStarShapedContinuousSolution}
    g \in \argmin \left\{ \langle J'(f),v \rangle : v \in \LipO (\mathbb{S}^{d-1}),\, \int_{\mathbb{S}^{d-1}} f^{d-1} v = 0 \right\}.
\end{equation}
then one has the existence of a solution c.f. \cite{DecHerHin21}.
This solution corresponds to the direction of steepest descent which fix the volume to first order.
For this constrained minimisation problem it is convenient to introduce a Lagrangian which may be given by
\begin{equation}
    L(g, \lambda) := %
    \langle J'(f),g \rangle
    -
    \lambda \int_{\mathbb{S}^{d-1}} f^{d-1} g
\end{equation}

Consider the problem to find a critical point of $L$.
One may see that the critical value for $\lambda$ is given by
\begin{equation}
\lambda^* =
\left( \int_{\mathbb{S}^{d-1}} f^{d-1} \right)^{-1}
\langle J'(f), 1 \rangle.
\end{equation}
We therefore see that
\begin{equation}
    \min \left\{ L(g,\lambda^*) : g \in \LipO(\mathbb{S}^{d-1}) \right\}
    =
    \min \left\{ \langle J'(f),g\rangle : g \in \LipO(\mathbb{S}^{d-1}),\, \int_{\mathbb{S}^{d-1}} f^{d-1}g =0 \right\}.
\end{equation}

In the case that $\mathcal{J}'(\Omega_f)$ is a measure, as in \eqref{eq:ShapeOptimisationMeasure}, then it is seen that $J'(f)$ is also a measure and we set
\begin{equation}
    \mu =%
    \left( J'(f) - \lambda^* f^{d-1} \right)^+
    \mbox{ and }
    \nu =%
    \left( J'(f) - \lambda^* f^{d-1} \right)^-
\end{equation}
where the superscript $+$ (resp. $-$) denotes the positive (resp. negative) part of a signed measure.
We see by the choice of $\lambda^*$ that $\int \dee \mu = \int \dee \nu$ and by applying Theorem \ref{thm:DualityResult} to a re-scaling of $\mu$ and $\nu$ (so they are probability measures), we have demonstrated Theorem \ref{thm:MainTheorem}.
\begin{theorem}\label{thm:MainTheorem}
    Let $d$ be the intrinsic metric on $\mathbb{S}^{d-1}$ and $f \colon \mathbb{S}^{d-1}\to \R$ with $f>0$ be Lipschitz.
    Suppose that $J'(f) \in \mathcal{M}(\mathbb{S}^{d-1})$ and let
    \begin{equation}
        \lambda^* = \left( \int_{\mathbb{S}^{d-1}} f^{d-1} \right)^{-1} \langle J'(f),1\rangle,
    \end{equation}
    \begin{equation}
        \mu = \left( J'(f) + \lambda f^{d-1} \right)^+ \mbox{ and } \nu = \left( J'(f) + \lambda f^{d-1} \right)^-,
    \end{equation}
    then $\int \dee \mu = \int \dee \nu =: \beta$ and one has the following duality type result:
    \begin{equation}\label{eq:TheLink}
    \begin{split}
        -\min& \left\{ \langle J'(f), v\rangle : v \in \LipO(\mathbb{S}^{d-1}),\, \int_{\mathbb{S}^{d-1}} f^{d-1} g =0 \right\}
        \\
        =&
        \\
        \min& \left\{ \int d(\omega,\omega') \dee \gamma (\omega,\omega') : \gamma \in \mathcal{P}(\mathbb{S}^{d-1}\times \mathbb{S}^{d-1}),\, \beta (\pi_x)_\# \gamma = \mu, \, \beta(\pi_y)_\# \gamma = \nu \right\}.
    \end{split}
    \end{equation}
\end{theorem}

\subsection{Numerical analysis in a star-shaped setting}
Before we provide numerical experiments, we wish to approximate the steepest descent for this star-shaped setting.
We restrict to the case $d=2$.
This restriction is useful to ensure that the interpolation operator on linear finite elements is non-expansive in the Lipschitz semi-norm.
The result we give will closely use the methods of \cite{Bar20} with slightly more general data.

Let us now settle notation.
By $\mathcal{S}_h$ we denote linear finite elements on a division $\mathbb{S}^{1}$ which has maximal element length $h$.
Denote by $I_h\colon W^{1,\infty}(\mathbb{S}^{1}) \to \mathcal{S}_h$ the Lagrange interpolation.
A discrete direction of steepest descent is given by
\begin{equation}\label{eq:findDiscreteDirectionOfDescent}
    g_h \in \argmin \left\{ \langle J'(f),v_h\rangle : v_h \in \mathcal{S}_h \cap \LipO(\mathbb{S}^1),\, \int_{\mathbb{S}^1}v_h f =0 \right\}.
\end{equation}
It is clear that a solution exists by the same arguments as in the continuous case \cite{DecHerHin21}.
In practice, we will approximate $g_h$ using the Sinkhorn algorithm outlined in Section \ref{sec:sinkhorn} with a post-processing given in \cite[Section 3.3]{DecHerHin21} to ensure the solution is Lipschitz-$1$.
Since we are considering the case $d=2$, the formula provided in \cite[Section 3.2]{DecHerHin21} could have been used.

The numerical analysis result will not say that $g-g_h$ is small in a metric, however it will state that the discrete direction of steepest descent provides almost as much descent as the continuous direction of steepest descent, this result is rather similar to Propositions \ref{prop:pApprox} and \ref{prop:ViscousApprox}.
\begin{proposition}\label{prop:NumericalAnalysis}
    Suppose that $J'(f)$ has representation as a measure as in \eqref{eq:ShapeOptimisationMeasure}, let $g \in W^{1,\infty}$ be as in \eqref{eq:findStarShapedContinuousSolution}, let $g_h \in W^{1,\infty}$ be as in \eqref{eq:findDiscreteDirectionOfDescent}.
    Then there is $C>0$ independent of $h$ such that
    \begin{equation}
        \langle J'(f), g \rangle
        \leq
        \langle J'(f), g_h \rangle
        \leq
        \langle J'(f),g \rangle +  C h \|J'(f)\|_{\mathcal{M}}.
    \end{equation}
\end{proposition}
\begin{proof}
Since $g_h \in \LipO(\mathbb{S}^1)$, it holds that
\begin{align}
    0
    \leq
    \langle J'(f), g_h - g\rangle
    =
    \langle J'(f), g_h - I_h g + \alpha \rangle + \langle J'(f), I_h g -\alpha - g\rangle,
\end{align}
where $\alpha = \frac{\int_{\mathbb{S}^1}f I_hg }{\int_{\mathbb{S}^1} f}$ ensures that $\int_{\mathbb{S}^1} (I_h g -\alpha)f =0$.
Furthermore, it holds that $\| \nabla_T I_h g \|_{L^\infty(\mathbb{S}^1)} \leq \| \nabla_T g \|_{L^\infty(\mathbb{S}^1)}$, therefore
\begin{equation}
    \left\langle J'(f), g_h - I_h g + \alpha \right\rangle \leq 0,
\end{equation}
To control the other term, we use that $\langle J'(f), I_h g - g\rangle \leq \| J'(f)\|_{\mathcal{M}} \| I_h g -g \|_{C^0} \leq C h\| J'(f)\|_{\mathcal{M}}$, where the second inequality is by standard interpolation results and that $g \in \LipO(\mathbb{S}^1)$.
Finally, since $\int_{\mathbb{S}^1}f g =0$, it follows that there is $C>0$ independent of $h$ such that $|\int_{\mathbb{S}^1}f I_hg| \leq Ch$.
\end{proof}
Of course, we have not taken into account that in a discrete setting, one does not have precise knowledge of $J'(f)$.
The approximation does not pose much of a difference to the proof.
The approximation of the shape derivative is shown in \cite{HPS15} in the setting of generic domains, rather than star-shaped domains.

\section{Numerical Experiments}\label{sec:Experiments}
We now provide some demonstrative numerical experiments.
These experiments will focus on the star-shaped setting and are conducted using DUNE \cite{duneReference}, in particular the DUNE Python bindings \cite{DunePython1,DunePython2}.
The methodology for these experiments differs from that given in \cite{DecHerHin21}.
In particular, the grid for $\mathbb{S}^1$ is constructed from the vertices on the boundary of the mesh which approximates the ball.

Before we conduct any Shape Optimisation experiments, we begin with a numerical verification of Proposition \ref{prop:pApprox} in the star-shaped setting.
After this, we provide a verification of Proposition \ref{prop:NumericalAnalysis}.

\subsection{Convergence as \texorpdfstring{$p \to \infty$}{p to infinity}}\label{sec:exp:pToInfinity}
For this, we set the 'shape derivative' to be given by the indicator-like function $\chi$ which is given by
\begin{equation}
    \chi(\theta) = 
    \begin{cases}
        0.1 & \theta \in [0,\pi),\\
        -0.1 & \theta \in [\pi,2\pi).
    \end{cases}
\end{equation}
We now approximate via FEM and an inexact Newton method
\begin{equation}\label{eq:numericalTestPToInfty}
    g_p \in \argmin \left\{ \int_{\mathbb{S}^1} {\chi g} + \frac{1}{p} \int_{\mathbb{S}^1}|g'|^p : g \in \mathcal{S}_h,\, \int_{\mathbb{S}^1} g = 0\right\}
\end{equation}
for a selection of finite $p$ as well as the solution to the problem
\begin{equation}\label{eq:numericalTestP=Infty}
    g_\infty \in \argmin \left\{ \int_{\mathbb{S}^1} {\chi g} : g \in \mathcal{S}_h\cap \LipO(\mathbb{S}^1),\, \int_{\mathbb{S}^1} g = 0\right\}.
\end{equation}
In Figure \ref{fig:p_experiment}, we see the graphs of these functions.
\begin{figure}
    \centering
    \includegraphics[width=.75\linewidth]{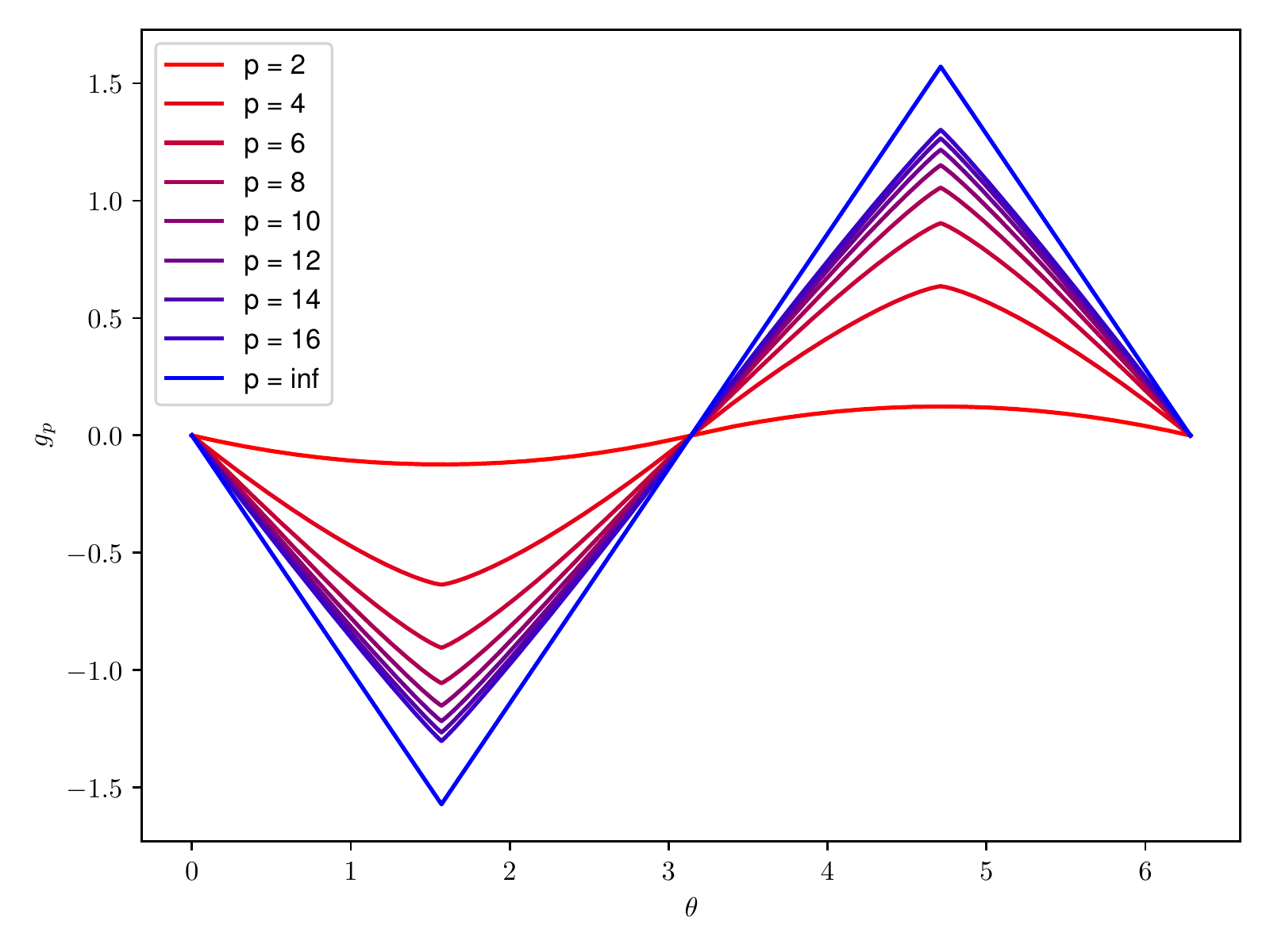}
    \caption{Graphs of $g_p$ for $p \in \{ 2,4,6,8,10,12,14,16, \infty\}$, where $g_p$ are as in \eqref{eq:numericalTestPToInfty} or \eqref{eq:numericalTestP=Infty}.
    For $p>2$, it is seen that the solutions have discontinuity in the first derivatives.
    }
    \label{fig:p_experiment}
\end{figure}
In this setting, the exact minimiser, i.e. the limit as $h\to 0$ of $g_\infty$ as in \eqref{eq:numericalTestP=Infty} satisfies $ \int_{\mathbb{S}^1} g_\infty \chi \to -\frac{\pi^2}{20}$, therefore we may compare the values $\int_{\mathbb{S}^1}g_p \chi$ to this.
This is tabulated in Table \ref{tab:p_experiment}.
\begin{table}
    \centering
    \begin{tabular}{c|c|c}
        $p$ &   $\frac{\pi^2}{20} + \int \chi g_p $    %
        &   ${\rm EOC}$
\\
\hline
2   &   0.441803289 %
&   --    \\
4   &   0.265292381  %
&   -0.735820919  \\
6   &   0.183709452  %
&   -0.906310037  \\
8   &   0.139958996  %
&   -0.945508853  \\
10  &   0.11291857  %
&   -0.962082737  \\
12  &   0.094596191  %
&   -0.971084987  \\
14  &   0.081373914  %
&   -0.976722648  \\
16  &   0.071388038  %
&   -0.980476769  \\
$\infty$  &   $4.13122872577886\times 10^{-8}$  %
&   --  \\
    \end{tabular}
    \caption{The energy $\frac{\pi^2}{20} + \int_{\mathbb{S}^1} g_p \chi$ for $g_p$ as in \eqref{eq:numericalTestPToInfty} or \eqref{eq:numericalTestP=Infty}.
    It is seen that this appears to converge like the expected $\frac{1}{p}$.
    }
    \label{tab:p_experiment}
\end{table}

\subsection{Convergence as \texorpdfstring{$h \to 0$}{h to 0}}
We will consider two experiments where $h \to 0$.
To begin with we have a manufactured problem as in the case $p\to \infty$.
The second experiment, which will appear later, will investigate the convergence of the first direction of steepest descent for a Shape Optimisation problem.

\label{sec:exp:manufactured:hTo0}
Let us fix the measure $\tilde{\mu} := \sum_{i=0}^3 \left( \delta_{0.05+i} - \delta_{\pi+0.05+i} \right)$, where $\delta_{\theta}$ is the Dirac delta at $\theta$.
We then find
\begin{equation}\label{eq:numericalTest_manufactured_hToZero}
    g_h \in \argmin \left\{ \int_{\mathbb{S}^1} g \dee \tilde{\mu} :  g \in \mathcal{S}_h\cap \LipO(\mathbb{S}^1),\, \int_{\mathbb{S}^1} g = 0\right\}.
\end{equation}
In Figure \ref{fig:h_experiment_manufactured}, the graphs of the approximations of the minimiser are given.
\begin{figure}
    \centering
    \includegraphics[width = .75\linewidth]{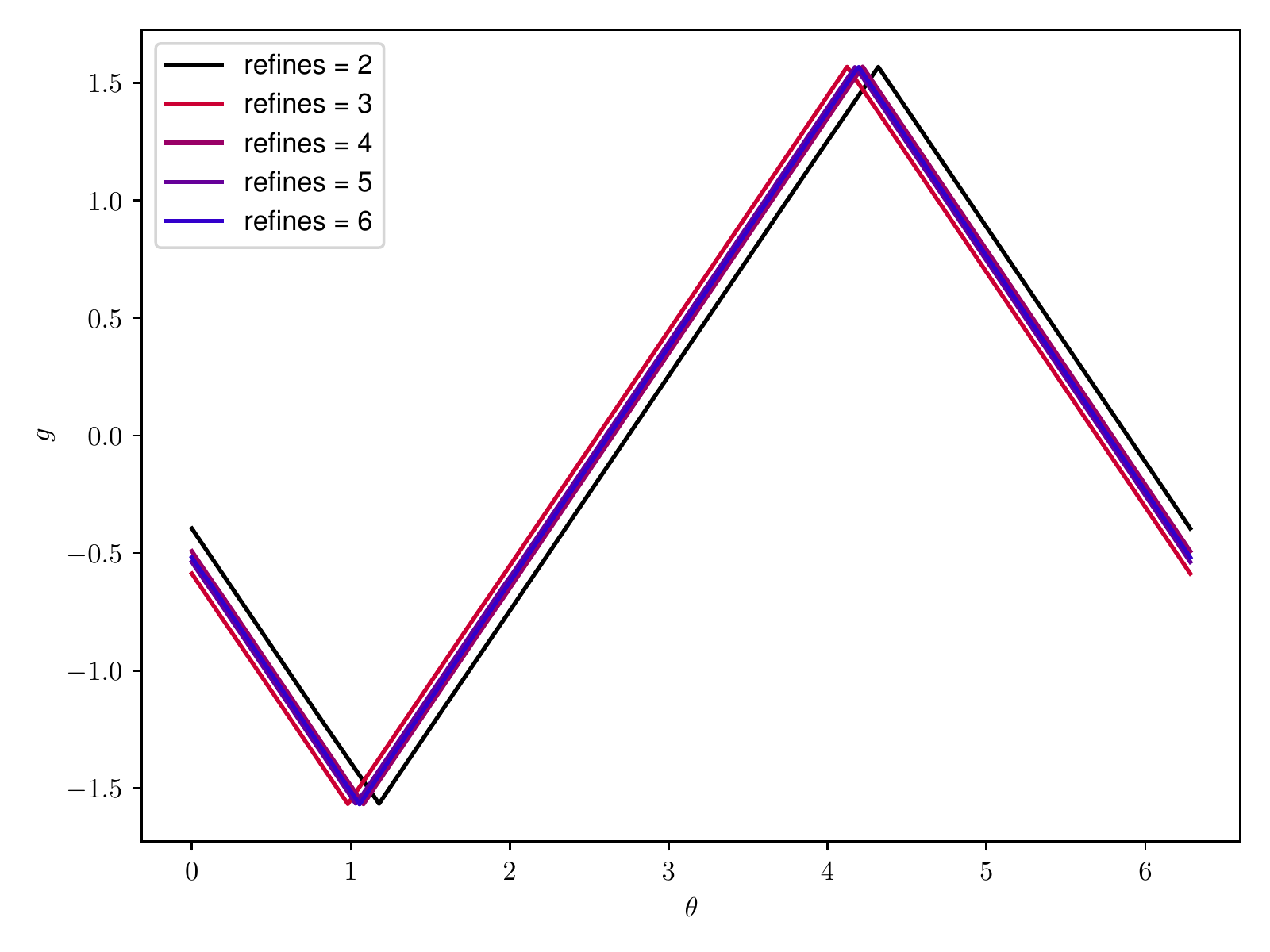}
    \caption{The functions $g_h$ as in \eqref{eq:numericalTest_manufactured_hToZero} for a selection of $h$ values.}
    \label{fig:h_experiment_manufactured}
\end{figure}
We know the exact solution has energy $4-3\pi$, this value is compared to those obtained in the experiments.
This is tabulated in Table \ref{tab:h_experiment_manufactured}.
\begin{table}
    \centering
    \begin{tabular}{c|c|c|c}
        $h$ &   $\int_{\mathbb{S}^1} g_h \dee \tilde{\mu}$   &   $3\pi -4 + \int_{\mathbb{S}^1} g_h \dee \tilde{\mu}$ &   ${\rm EOC}$
        \\ \hline
        0.3926990816987246  &   -5.165503580597086  &   0.2592743801722932  &   --  \\
        0.19634954084936274 &   -5.285169111921592  &   0.13960884884778757 &   0.8930892683900136  \\
        0.09817477042468248 &   -5.361914059812256  &   0.06286390095712324 &   1.1510866819541183  \\
        0.04908738521234213 &   -5.381993219295463  &   0.0427847414739162  &   0.5551354267772528  \\
        0.024543692606171064    &   -5.409984309637617  &   0.014793651131762786    &   1.5321182139092757  \\

    \end{tabular}
    \caption{The energy $\int_{\mathbb{S}^1} g_h \dee \tilde{\mu}$ for a selection of $h$ values, where $g_h$ as in \eqref{eq:numericalTest_manufactured_hToZero}.
    On average, the energy appears to converge like the expected order $h$.
    }
    \label{tab:h_experiment_manufactured}
\end{table}

\subsection{Shape Optimisation experiments}
In these experiments, we will consider the state problem to be given by the Poisson problem which appears in Example \ref{ex:PoissonExample}, we will consider two different energies, the tracking-type energy which appears in \ref{ex:PoissonExample} and the energy
\begin{equation}\label{eq:linearEnergy}
    J^*_\Omega(v) := \int_\Omega v-Z
\end{equation}
for some given function $Z$.
We will approximate the shape derivative by replacing the solutions to the state and adjoint equations by solutions to the FEM state and adjoint equations in $J'(f)$.

The experiments are similar to those conducted in \cite{DecHerHin21}, making use of an Armijo line search and projecting to the correct volume at each step.
The domain is constrained to have area $4\pi$.
The Armijo step length we consider is given by, at state $f^k$ with update direction $g^k$, finding
\begin{equation}
    \max \left\{ 2^{-k} : k\geq 1, f^k+2^{-k}g^k >0,\,  J(f^k+2^{-k}g^k) - J(f^k) \leq \gamma 2^{-k} \langle J'(f^k),g^k\rangle \right\}
\end{equation}
for some $\gamma>0$.
In calculating the Armijo step length, we set $\gamma = 10^{-3}$, here we have restricted the largest admissible step to be $\frac{1}{2}$.

In our experiments, we will consider the classical Hilbertian method which is effectively using the Laplace-Beltrami operator as well as the case $p$-Laplace-Beltrami operator $p=4$, that is to say we seek minimisers
\begin{equation}
    g_p \in \argmin \left\{ \langle J(f),g\rangle + \frac{1}{p} \| \nabla_T g\|_{L^p(\mathbb S^{d-1})}^p : g \in W^{1,p}(\mathbb S^{d-1}), \int_{\mathbb S^{d-1}} f^{d-1}g = 0 \right\}.
\end{equation}
The solution for $p=4$ is calculated using an inexact Newton method.
Let us recall that, for the Lipschitz approximation, we will exploit the relationship to Optimal Transport we have shown and use a Sinkhorn-Knopp algorithm to approximate the solution.
In order to fairly compare the different descent functions, the functions are normalised to have a Lipschitz constant of $1$.

\subsubsection{Experiment with no PDE}\label{sec:exp:NoPDE}
For this experiment we effectively do not have a PDE state equation, this is an easy method to manufacture solutions which should have corners.
In the state problem we consider that $F=0$ which therefore gives that $u=0$ in $\Omega$.
For this experiment, we consider the energy given by \eqref{eq:linearEnergy}, which, by our choice of $F$ means that we are effectively considering
\begin{equation}
    \Omega \mapsto \int_\Omega-Z.
\end{equation}
We will take $Z(x) = - |x_1| - |x_2|$ which should lead to a minimiser of the square $(-\pi,\pi)^2$ which has been rotated by an angle of $\frac{\pi}{4}$.
This should have energy $-\frac{16}{\pi}$.

For the initial state, we take $f$ which represents the square $(-\sqrt{\pi},\sqrt{\pi})^2$, this is shown in Figure \ref{fig:NoPDE:InitialDomains}.
\begin{figure}
    \centering
    \includegraphics[width = 0.5\linewidth]{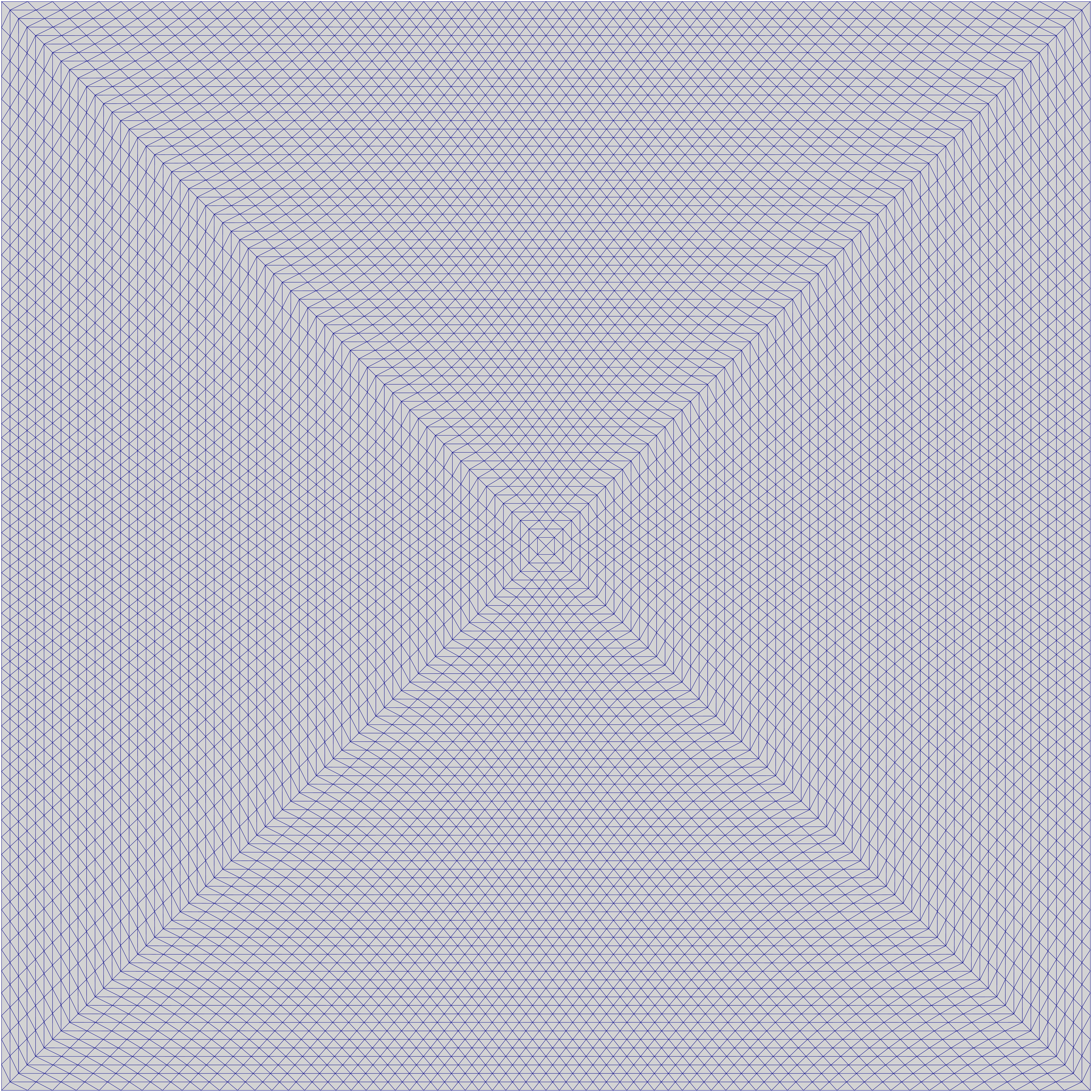}
    \caption{Initial domain for the experiment in Section \ref{sec:exp:NoPDE}.}
    \label{fig:NoPDE:InitialDomains}
\end{figure}
After 50 steps, we obtain the domains shown in Figure \ref{fig:NoPDE:FinalDomains}.
\begin{figure}
    \centering
    \includegraphics[width = 0.31\linewidth]{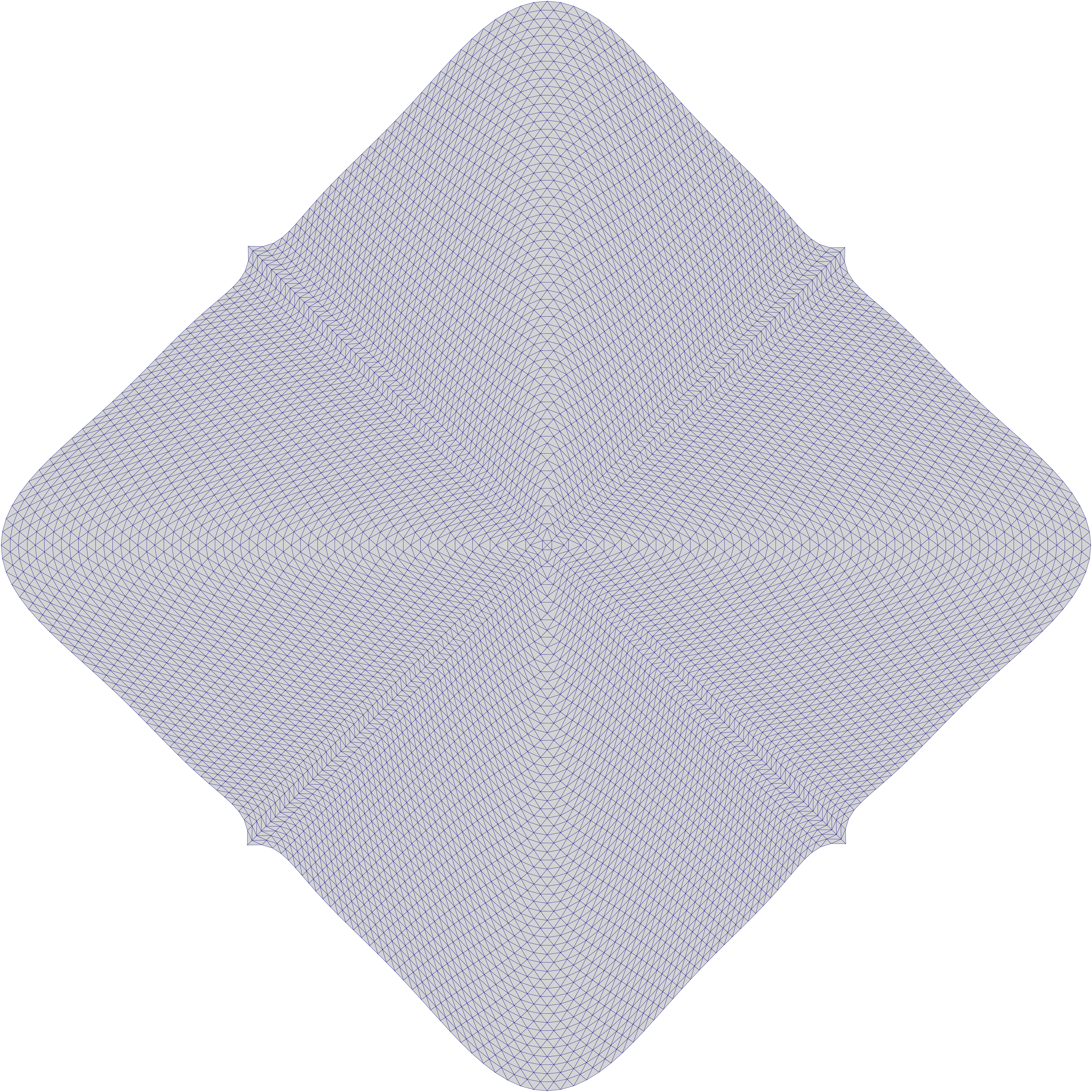}\hfill
    \includegraphics[width = 0.31\linewidth]{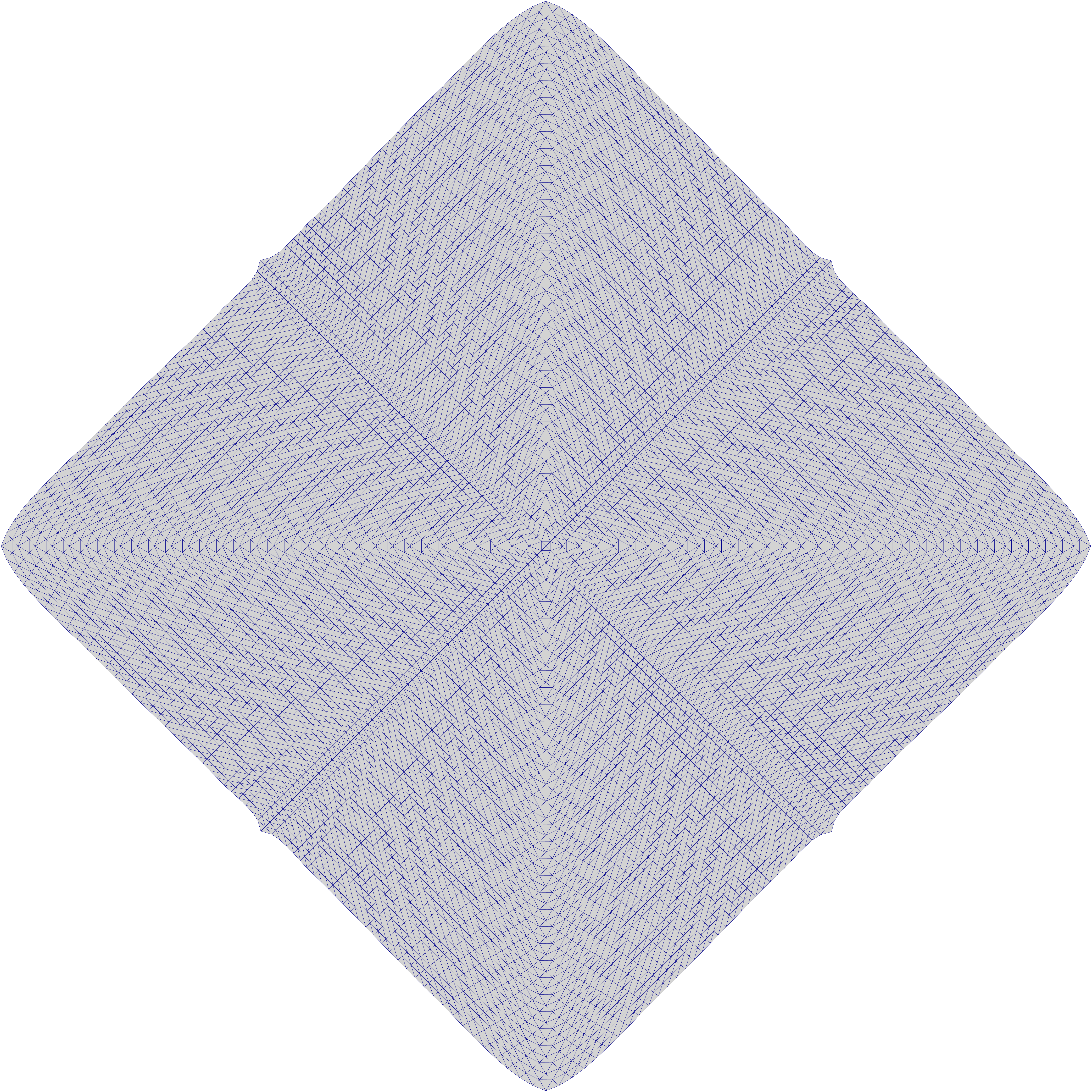}\hfill
    \includegraphics[width = 0.31\linewidth]{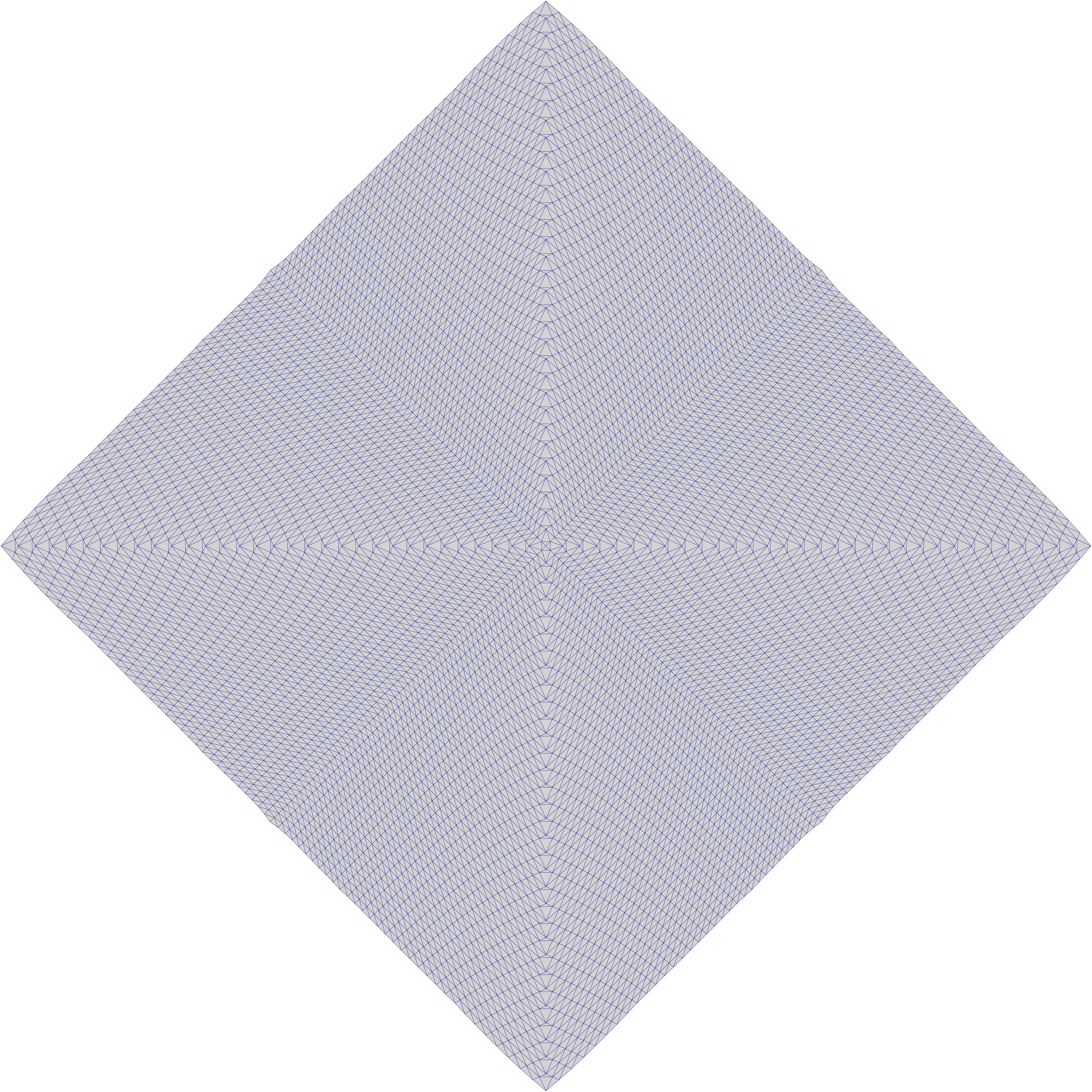}
    \caption{Domains after 50 steps of the Armijo algorithm using the, from left to right, $2$,$4$ and Sinkhorn methods in the experiment in Section \ref{sec:exp:NoPDE}.
    The Sinkhorn method is the only method which appears to create the new corners and remove the corners from the initial triangulation.
    }
    \label{fig:NoPDE:FinalDomains}
\end{figure}
The graph of the energy
is given in Figure \ref{fig:NoPDE:Graphs}.
\begin{figure}
    \centering
    \includegraphics[width = 0.75\linewidth]{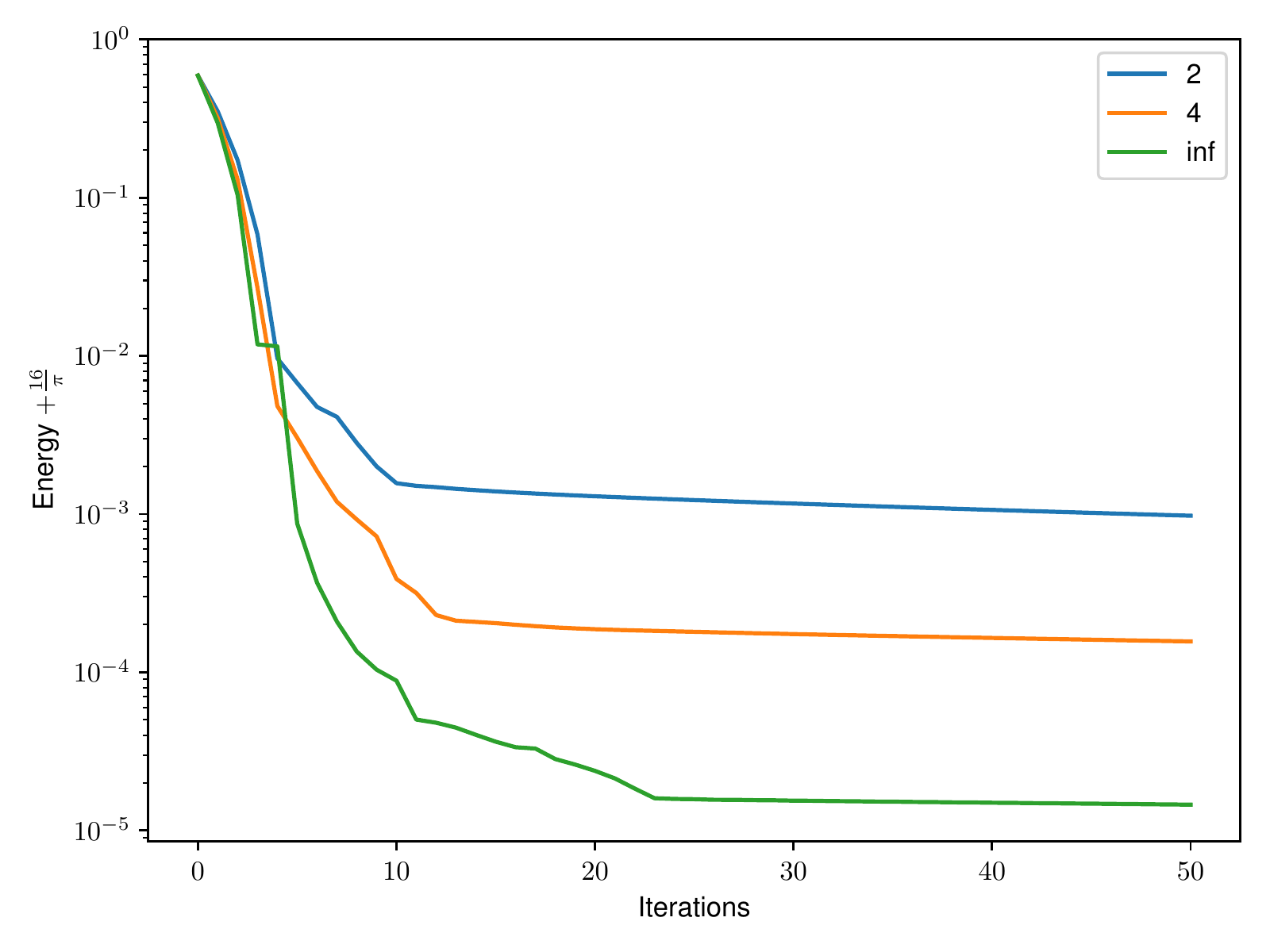}
    \caption{Graph of the energy
    for the experiment in Section \ref{sec:exp:NoPDE}.
    The Sinkhorn method is energetically outperforming the $p=2$ and $p=4$ methods.
    }
    \label{fig:NoPDE:Graphs}
\end{figure}

\subsubsection{Experiment with Laplace side-constraint}\label{sec:exp:Laplace}
For this experiment, we consider $F(x) = 1- |x|^2$.
We take the quadratic energy which appears in \eqref{eq:quadraticEnergy} with $Z(x) = 1 - |x|^2$.
We set the initial shape to be the square $(-\sqrt{\pi},\sqrt{\pi})^2$

To begin with, we investigate the convergence of the steepest descent of the first step on subsequently refined grids.
Here, we do not know the exact value of the descent given by the steepest descent, as such we compare the value from the finest grid to that obtained in the other experiments.
This is tabulated in Table \ref{tab:h_experiment_shape}.
\begin{table}
    \centering
    \begin{tabular}{c|c|c|c}
         $h$    &   $\langle J'(f)_h,g_h\rangle$    &   ${\rm E}_{\langle J'(f)_h,g_h\rangle}$  &   ${\rm EOC}$
         \\\hline
    0.3926990816987246  &   -15.669222693350143 &   3.0716367046556883  &   --  \\
    0.19634954084936274 &   -18.04371205197436  &   0.6971473460314712  &   2.13947207737662  \\
    0.09817477042468248 &   -18.603170783008334 &   0.13768861499749718 &   2.3400543372054563  \\
    0.04908738521234213 &   -18.7129553820854   &   0.02790401592043068 &   2.3028645994071724    \\
    0.024543692606171064    &   -18.734807939237772 &   0.006051458768059348    &   2.205117902043274 \\
    0.01227184630308642 &   -18.739684589664556 &   0.0011748083412754795   &   2.36485754707719  \\
    0.006135923151544098    &   -18.74085939800583  &   --  &   --
    \end{tabular}
    \caption{The energy $\langle J'(f)_h,g_h\rangle$ for a selection of $h$ values and $g_h$ being the first direction of steepest descent for the problem detailed in Section \ref{sec:exp:Laplace}.
    We see that the energy converges faster than the expected order of $h$, this is attributed to the approximation of $J'(f)$.
    }
    \label{tab:h_experiment_shape}
\end{table}

Let us now turn to the Shape Optimisation itself.
We expect the minimiser to be a ball of radius $2$ with centre at the origin.
This has energy $\frac{61 \pi}{15}$.

For the initial state, we again take $f$ which represents the square $(-\sqrt{\pi},\sqrt{\pi})^2$, this is shown in Figure \ref{fig:Laplace:InitialDomains}.
\begin{figure}
    \centering
    \includegraphics[width = 0.5\linewidth]{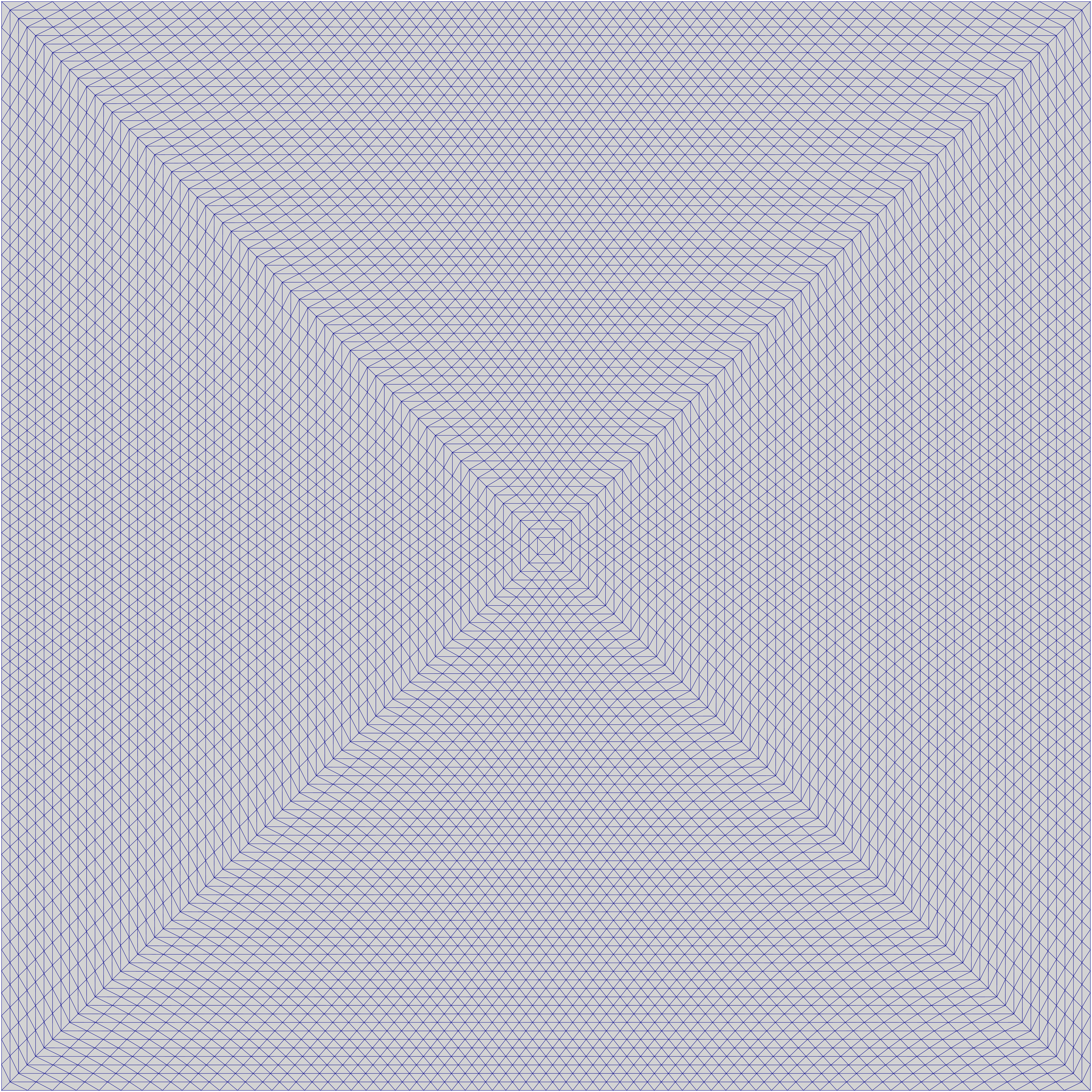}
    \caption{Initial domain for the experiment in Section \ref{sec:exp:Laplace}.}
    \label{fig:Laplace:InitialDomains}
\end{figure}
After 50 steps, we obtain the domains shown in Figure \ref{fig:Laplace:FinalDomains}.
\begin{figure}
    \centering
    \includegraphics[width = 0.31\linewidth]{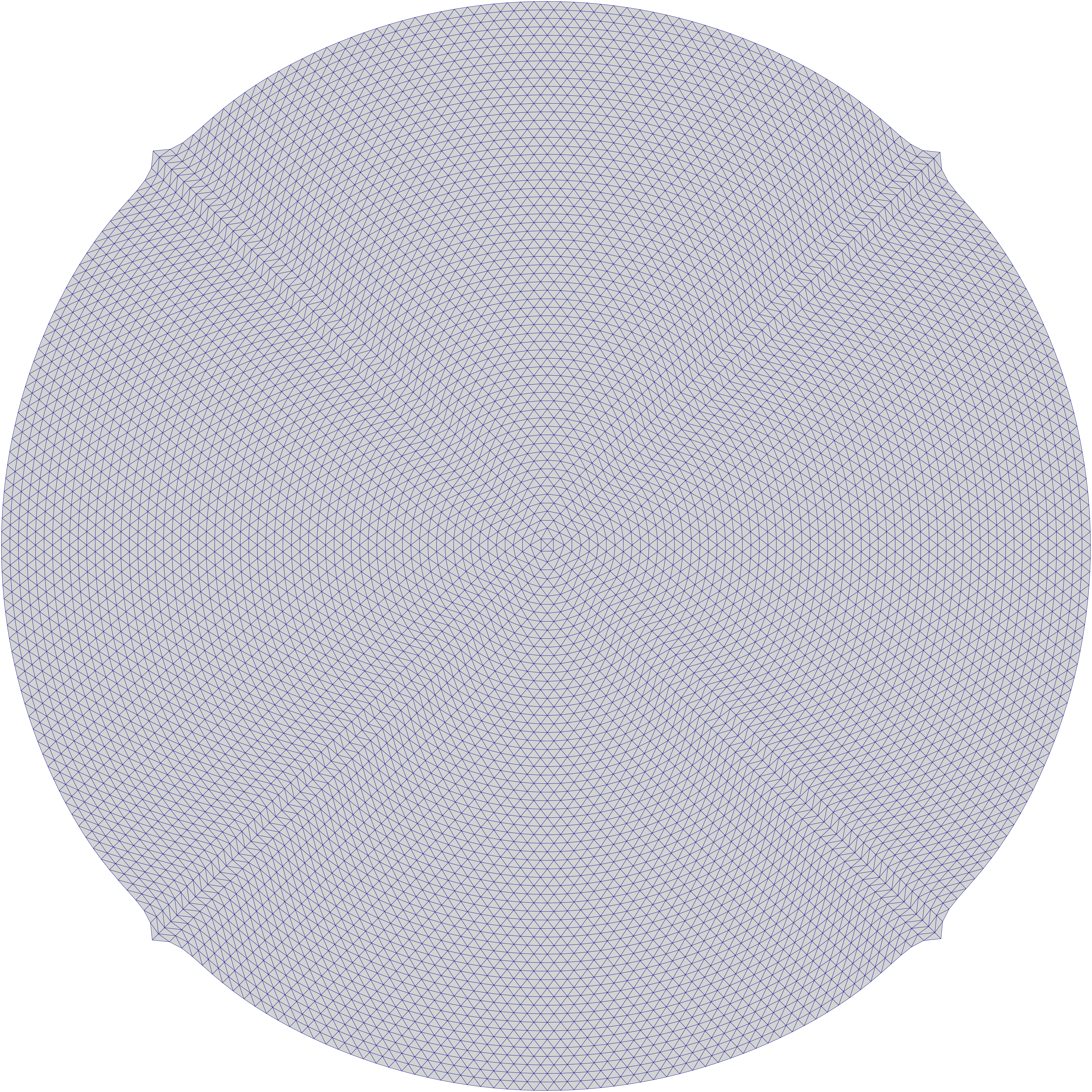}\hfill
    \includegraphics[width = 0.31\linewidth]{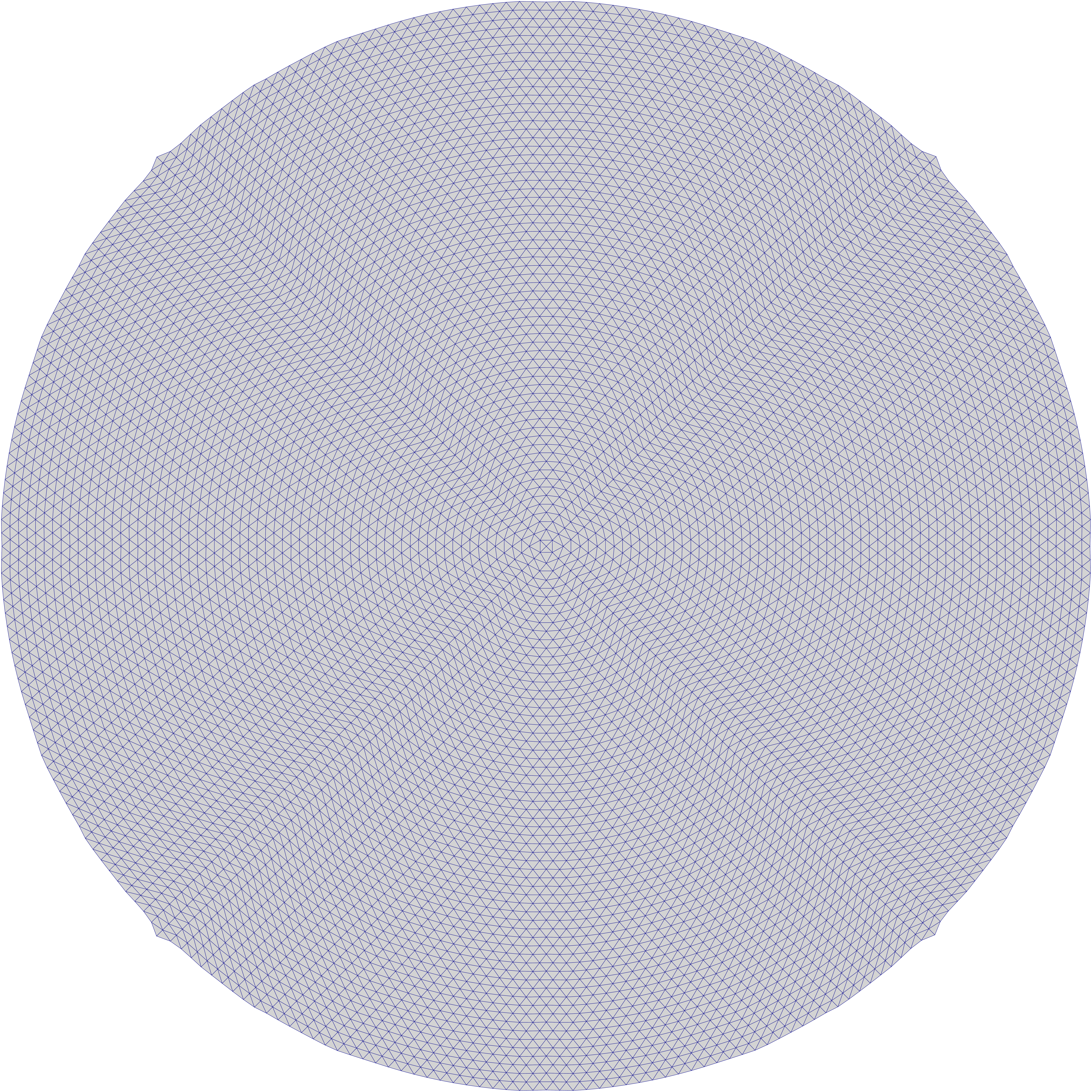}\hfill
    \includegraphics[width = 0.31\linewidth]{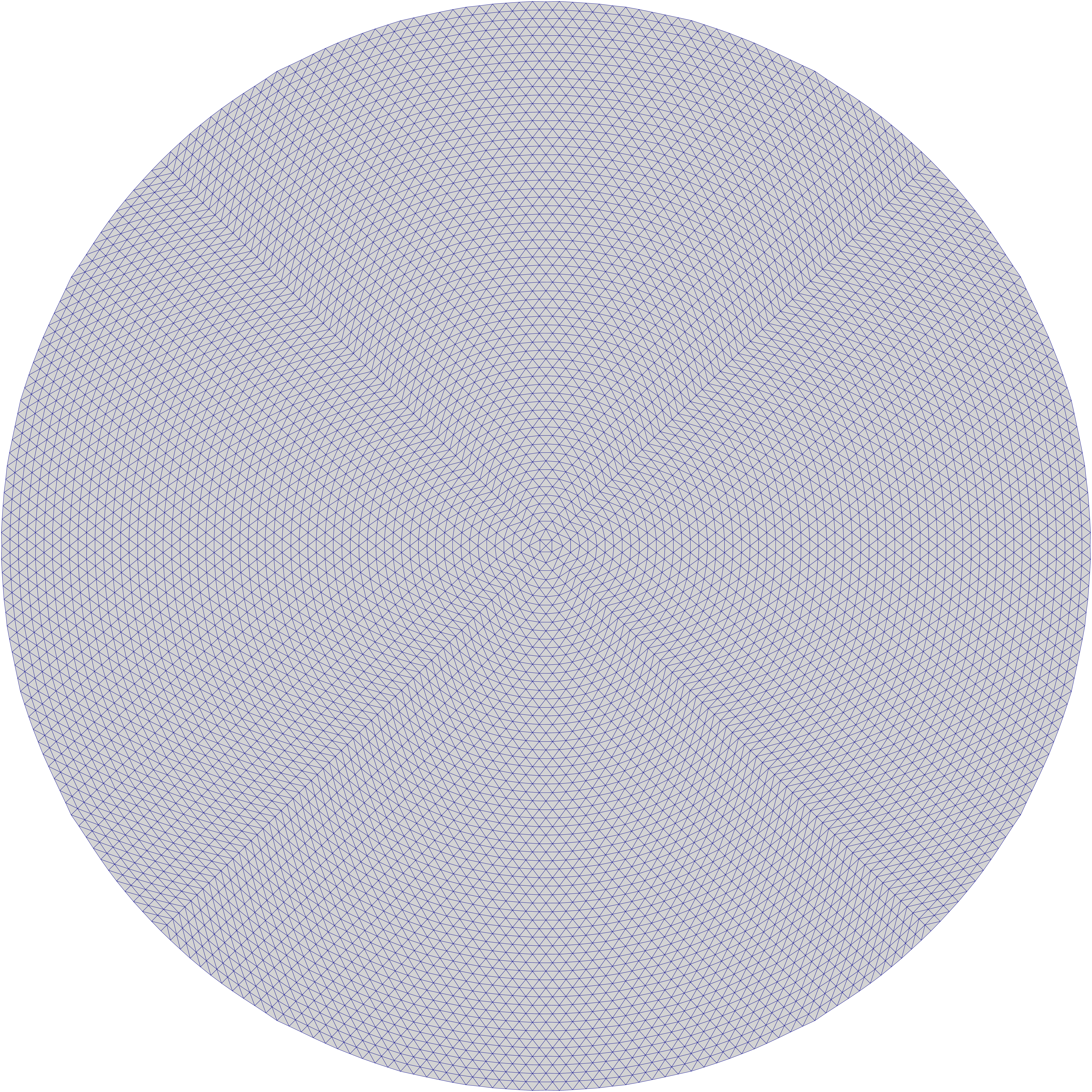}
    \caption{Domains after 50 steps of the Armijo algorithm using the, from left to right, $2$,$4$ and Sinkhorn methods in the experiment in Section \ref{sec:exp:Laplace}.
    The Sinkhorn method is the only method which appears to completely remove the corners from the initial triangulation.
    }
    \label{fig:Laplace:FinalDomains}
\end{figure}
The graph of the energy
is given in Figure \ref{fig:Laplace:Graphs}.
\begin{figure}
    \centering
    \includegraphics[width = 0.75\linewidth]{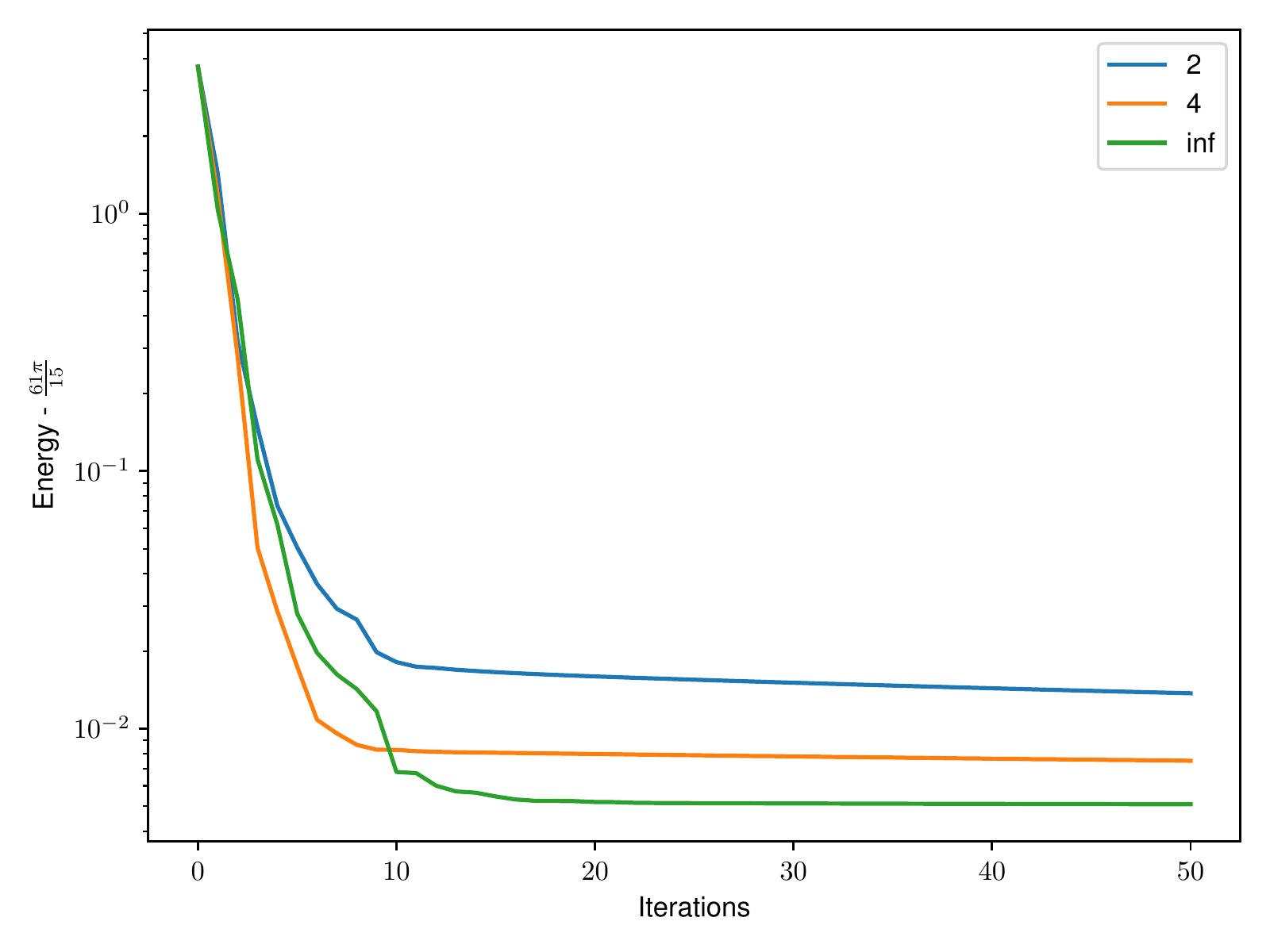}
    \caption{Graph of the energy
    for the experiment in Section \ref{sec:exp:Laplace}.
    We see that the Sinkhorn method is again energetically outperforming the $p=2$ and $p=4$ methods.
    }
    \label{fig:Laplace:Graphs}
\end{figure}

\subsubsection{Remarks on the experiments}
We see that the $W^{1,\infty}$ strategy using the Sinkhorn method is outperforming the other two methods given, both in terms of energy descent and the resulting shapes.
Let us comment on the computation time - the $p=2$ and Sinkhorn methods were comparable in computational time whereas the $p=4$ method was a little bit slower.

\section*{Acknowledgements}
The author would like to thank Klaus Deckelnick, Charlie Elliott, and Michael Hinze for helpful suggestions and insightful discussions during the preparation of this work.
This work is part of the project P8 of the German Research Foundation Priority Programme 1962, whose support is gratefully acknowledged by the author.

\printbibliography

\end{document}